\documentclass[9pt]{amsart}
\usepackage{amssymb}
\usepackage{hyperref}
\usepackage[all]{xy}
\usepackage{enumerate}
\usepackage{ marvosym }
\usepackage{ bbold }
\usepackage{ tensor }
\usepackage{url}

\usepackage{bbm}

\usepackage[dvipsnames]{xcolor}
\usepackage{graphicx}
\usepackage{extpfeil}

\DeclareMathOperator{\Gra}{{Gr}}

\DeclareMathOperator{\reg}{{reg}}

\usepackage{tikz-cd}

\newcommand{\bC}{{\mathbb C}}

\newcommand{\bE}{{\mathbb E}}

\newcommand{\bG}{{\mathbb G}}

\newcommand{\bL}{{\mathbb L}}

\newcommand{\bQ}{{\mathbb Q}}

\newcommand{\cE}{{\mathcal E}}
\newcommand{\cF}{{\mathcal F}}
\newcommand{\cG}{{\mathcal G}}
\newcommand{\cH}{{\mathcal H}}
\newcommand{\cI}{{\mathcal I}}

\newcommand{\cL}{{\mathcal L}}
\newcommand{\cM}{{\mathcal M}}
\newcommand{\cN}{{\mathcal N}}
\newcommand{\cO}{{\mathcal O}}
\newcommand{\cP}{{\mathcal P}}

\newcommand{\cY}{{\mathcal Y}}
\newcommand{\cZ}{{\mathcal Z}}

\newcommand{\Ran}{Ran_{X}}
\newcommand{\Ranp}{{\on{Ran}}}

\newcommand{\fs}{{\mathfrak s}}

\DeclareMathOperator*{\colim}{colim}

\newcommand{\nc}{\newcommand}
\nc{\renc}{\renewcommand}
\nc{\ssec}{\subsection}
\nc{\sssec}{\subsubsection}
\nc{\on}{\operatorname}

\nc\Gr{\on{Gr}}

\nc\Fl{\on{Fl}}

\newtheorem{thm}[subsubsection]{Theorem}

\newtheorem{conj}{Conjecture}
\DeclareMathOperator{\DGCat}{{DGCat}}
\DeclareMathOperator{\AGCat}{{AGCat}}

\DeclareMathOperator{\QCoh}{{QCoh}}

\DeclareMathOperator{\oblv}{{oblv}}

 \DeclareMathOperator{\Bun}{{Bun}}

\DeclareMathOperator{\irr}{{irr}} \DeclareMathOperator{\Eis}{{Eis}}

\DeclareMathOperator{\LocSys}{{LS}}

\DeclareMathOperator{\Sym}{{Sym}}

\newcommand{\limto}{{\displaystyle\lim_{\longrightarrow}}}
\newcommand{\rightlim}{\mathop{\limto}}

\newcommand{\leftlim}{\mathop{\displaystyle\lim_{\longleftarrow}}}
\newcommand{\limfromn}{\leftlim\limits_{\raise3pt\hbox{$n$}}}
\newcommand{\limton}{\rightlim\limits_{\raise3pt\hbox{$n$}}}

\newcommand{\rightlimit}[1]{\mathop{\lim\limits_{\longrightarrow}}\limits%
	_{\raise3pt\hbox{$\scriptstyle #1$}}}

\newcommand{\leftlimit}[1]{\mathop{\lim\limits_{\longleftarrow}}\limits%
	_{\raise3pt\hbox{$\scriptstyle #1$}}}

\DeclareMathOperator{\Id}{{Id}}

\DeclareMathOperator{\End}{{End}} 
\DeclareMathOperator{\Hom}{{Hom}}

 \DeclareMathOperator{\id}{{id}}
 \DeclareMathOperator{\Mmod}{{-mod}}

\DeclareMathOperator{\Nilp}{{Nilp}} \DeclareMathOperator{\op}{{op}}

\DeclareMathOperator{\Vect}{{Vect}}

\DeclareMathOperator{\GL}{{GL}}

\DeclareMathOperator{\Rep}{{Rep}}

\makeatletter

\newcommand{\Rmnum}[1]{\expandafter\@slowromancap\romannumeral #1@}
\makeatother

\newtheorem{pr}[subsubsection]{Proposition}
\newtheorem{lm}[subsubsection]{Lemma}

\newtheorem{cor}[subsubsection]{Corollary}

\newtheorem{cnstr}[subsubsection]{Construction}

\newtheorem{df}[subsubsection]{Definition}

\newtheorem{rem}[subsubsection]{Remark}

\newtheorem{ex}[subsubsection]{Example}
\newtheorem{ntn}[subsubsection]{Notation}

\numberwithin{equation}{section}

\newcommand{\Fun}{\operatorname{Fun}}

\DeclareMathOperator{\ad}{{ad}}

\DeclareMathOperator{\laxlim}{{laxlim}}

\DeclareMathOperator{\Alg}{{Alg}}
\DeclareMathOperator{\D}{{DMod}}

\DeclareMathOperator{\Conf}{{Conf}}

\DeclareMathOperator{\psid}{{Ps-Id}}
\DeclareMathOperator{\mir}{{Mir}}

\DeclareMathOperator{\cusp}{{cusp}}

\nc{\WFactAlg}{\on{WFactAlg}}
\nc{\FactAlg}{\on{FactAlg}}
\nc{\bFact}{\mathbf{Fact}}
\nc{\FactAlgCat}{\mathbf{FactAlgCat}}
\nc{\WFactAlgCat}{\mathbf{WFactAlgCat}}
\nc{\LWFactAlgCat}{\mathbf{LWFactAlgCat}}
\nc{\WFactModCat}{\on{-}\!\mathbf{WFactModCat}}
\nc{\LWFactModCat}{\on{-}\!\mathbf{LWFactModCat}}
\nc{\FactModCat}{\on{-}\!\mathbf{FactModCat}}
\nc{\WFactMod}{\on{-WFactMod}}
\nc{\FactMod}{\on{-FactMod}}
\nc{\weak}{{\on{weak}}}
\nc{\lax}{{\on{lax}}}
\nc{\str}{{\on{str}}}
\nc{\laxun}{{\on{laxun}}}
\nc{\strun}{{\on{strun}}}
\nc{\enh}{{\on{enh}}}

\nc{\ic}{{\on{IC}}}

\nc{\an}{{\on{an}}}
\nc{\re}{{\on{Re}}}
\nc{\Shv}{{\on{Shv}}}
\nc{\Tot}{{\on{Tot}}}
\nc{\eq}{{\on{equiv}}}
\nc{\rh}{{\on{RH}}}
\nc{\rs}{{\on{r.s.}}}
\nc{\hol}{{\on{hol}}}
\nc{\Kos}{{\on{Kos}}}
\nc{\SL}{{\on{SL}}}
\nc{\redu}{{\on{red}}}
\nc{\CC}{{\on{CC}}}
\nc{\Sph}{{\on{Sph}}}
\nc{\Av}{{\on{Av}}}
\nc{\temp}{{\on{Wh-temp}}}
\nc{\atemp}{{\on{wh-anti-temp}}}
\nc{\Whit}{{\on{Whit}}}
\nc{\irreg}{{\on{irreg}}}
\nc{\coeff}{{\on{coeff}}}
\nc{\loc}{{\on{loc}}}
\nc{\rel}{{\on{rel}}}
\nc{\lvl}{{\on{lvl}}}
\nc{\Sat}{{\on{Sat}}}
\nc{\ST}{{\on{StratTop}}}
\nc{\cshv}{{\on{cShv}}}
\nc{\Exit}{{\on{Exit}}}
\nc{\Enter}{{\on{Enter}}}
\nc{\LS}{{\on{LS}}}
\nc{\laxnat}{{\on{laxnat}}}
\nc{\Obj}{{\on{Obj}}}
\nc{\Poinc}{{\on{Poinc}}}
\nc{\CT}{{\on{CT}}}
\nc{\disj}{{\on{disj}}}
\nc{\co}{{\on{co}}}
\nc{\prsmall}{{\on{prsmall}}}
\nc{\Zas}{{\mathsf{Zas}}}
\nc{\ger}{{\on{Ge}}}
\nc{\fger}{{\on{FactGe}}}
\nc{\sic}{{\on{sc}}}
\nc{\triv}{{\on{triv}}}

\nc\ol{\overline}
\nc\ul{\underline}
\nc\wt{\widetilde}

\nc{\sA}{{\mathsf{A}}}
\nc{\sB}{{\mathsf{B}}}
\nc{\sC}{{\mathsf{C}}}
\nc{\sD}{{\mathsf{D}}}
\nc{\sF}{{\mathsf{F}}}
\nc{\sG}{{\mathsf{G}}}
\nc{\sH}{{\mathsf{H}}}
\nc{\sK}{{\mathsf{K}}}
\nc{\sM}{{\mathsf{M}}}
\nc{\sN}{{\mathsf{N}}}
\nc{\sO}{{\mathsf{O}}}
\nc{\sW}{{\mathsf{W}}}
\nc{\sQ}{{\mathsf{Q}}}
\nc{\sP}{{\mathsf{P}}}
\nc{\sR}{{\mathsf{R}}}
\nc{\sT}{{\mathsf{T}}}
\nc{\sZ}{{\mathsf{Z}}}
\nc{\sfi}{{\mathsf{i}}}
\nc{\sfj}{{\mathsf{j}}}
\nc{\sfp}{{\mathsf{p}}}
\nc{\sfq}{{\mathsf{q}}}
\nc{\sfs}{{\mathsf{s}}}
\nc{\sft}{{\mathsf{t}}}
\nc{\sr}{{\mathsf{r}}}
\nc{\sfk}{{\mathsf{k}}}
\nc{\sa}{{\mathsf{s}}}
\nc{\sg}{{\mathsf{g}}}
\nc{\sn}{{\mathsf{n}}}
\nc{\sh}{{\mathsf{h}}}
\nc{\sff}{{\mathsf{f}}}
\nc{\sfb}{{\mathsf{b}}}
\nc{\sfc}{{\mathsf{c}}}
\nc{\sfe}{{\mathsf{e}}}
\nc{\sd}{{\mathsf{d}}}
\nc{\sotimes}{\overset{!}\otimes}
\nc{\crit}{{crit}}

\nc{\BunPmt}{{\widetilde{\Bun}_{P^-}}}
\nc{\BunPmtpR}{{\widetilde{\Bun}_{P^-,\Ranp}}}
\nc{\BunPmtpZ}{{\widetilde{\Bun}_{P^-,\cZ}}}
\nc{\BunPmtpZp}{{\widetilde{\Bun}_{P^-,\cZ;}}}
\nc{\BunPmtpZsub}{{\widetilde{\Bun}_{P^-,\cZ^{\subseteq}}}}
\nc{\BunPmtpS}{{\widetilde{\Bun}_{P^-,S}}}
\nc{\BunPmtpx}{{\widetilde{\Bun}_{P^-,\ul{x}}}}
\nc{\BunNtpx}{{\widetilde{\Bun}_{N_P^-,\ul{x}}}}
\nc{\BunNbR}{{\ol{\Bun}_{N,\rho(\omega_X),\Ran}}}
\nc{\BunNbg}{{\ol{\Bun}_{N,\rho(\omega_X),\on{gen}}}}
\nc{\BunNbx}{{\ol{\Bun}_{N,\rho(\omega_X),\ul{x}}}}
\nc{\BunNbZ}{{\ol{\Bun}_{N,\rho(\omega_X),\cZ}}}
\nc{\BunNbZsub}{{\ol{\Bun}_{N,\rho(\omega_X),\cZ^\subseteq}}}
\nc{\BunNbS}{{\ol{\Bun}_{N,\rho(\omega_X),S}}}
\nc{\BunNbMg}{{\ol{\Bun}_{N(M),\rho(\omega_X),\on{gen}}}}
\nc{\BunNbMZ}{{\ol{\Bun}_{N(M),\rho(\omega_X),\cZ}}}
\nc{\BunNbMZsub}{{\ol{\Bun}_{N(M),\rho(\omega_X),\cZ^\subseteq}}}
\nc{\BunNbMR}{{\ol{\Bun}_{N(M),\rho(\omega_X),\Ranp}}}
\nc{\BunNbMgM}{{\ol{\Bun}_{N(M),\rho_M(\omega_X),\on{gen}}}}
\nc{\BunNbMZM}{{\ol{\Bun}_{N(M),\rho_M(\omega_X),\cZ}}}
\nc{\BunNbMZsubM}{{\ol{\Bun}_{N(M),\rho_M(\omega_X),\cZ^\subseteq}}}
\nc{\BunNbMRM}{{\ol{\Bun}_{N(M),\rho_M(\omega_X),\Ranp}}}

\begin{document}
	
	\title[Quantum Betti geometric Langlands functor]{Quantum Betti geometric Langlands functor}
	
	\author[E.~Bogdanova]{Ekaterina Bogdanova}
	\address{Harvard University,  USA}
	\email{ebogdanova@math.harvard.edu}
	
	\begin{abstract}
		We construct the quantum geometric Langlands functor in the Betti setting via Whittaker coefficients. We show that the functor is compatible with the 2-Fourier-Mukai equivalence between sheaves of categories over 2-stacks $\ger_{Z_G}$ and $\ger_{\pi_1(\check{G})}$, which classify gerbes on $X$ with respect to the center $Z_G$ of $G$ and algebraic fundamental group $\pi_1(\check{G})$ of $\check{G}$. 
	\end{abstract}
	
	\maketitle

	\tableofcontents
	
	\section{Introduction.}
	The goal of this paper is to construct the functor 
	$$\Shv_{\kappa, \Nilp}(\Bun_G) \xrightarrow[]{\bL_{\kappa}^{\operatorname{Betti}}} \int_X \Rep_{q}(\check{G})$$
	between the category of twisted automorphic sheaves with singular support in the global nilpotent cone to the topological factorization homology of $\Rep_{q}(\check{G})$. The existence of such functor was conjectured in \cite{BZN} and it is expected to be an equivalence. 
	
	At the limit $\kappa \rightarrow \kappa_{\operatorname{crit}}$ we recover the functor $\bL^{\operatorname{Betti}}$ constructed in \cite{GLCI}. However, our construction uses Whittaker coefficients instead of Hecke action, and thus works in the quantum context as well.

	\subsection{Betti quantum geometric Langlands theory.}
	
	The story of geometric Langlands correspondence begin with the following observation. Analogous to the space of automorphic forms, the category of automoprhic sheaves carries a family of commuting {\it Hecke operations} given by elements $V \in \Shv(\Gra_G)^{L^+G}$ for every point $x \in X$. Here $\Gra_G$ stands for the affine Grassmannian, i.e. the quotient $LG / L^+G$ of the loop group of $G$ by the arc group. 
	
	\begin{thm}\label{heckeeigen}
		For every irreducible $\check{G}$-local system $\sigma$ there exists a unique Hecke eigensheaf with eigenvalue $\sigma$.
	\end{thm}
	This statement is due to Deligne for $G=\GL_1$, Drinfeld for $G=\GL_2$, Frenkel-Gaitsgory-Vilonen for $G = \GL_n$, and Arinkin, Beraldo, Campbell, Chen, Faergeman, Gaitsgory, Lin, Raskin and  Rozenblyum for general $G$.
	
	\subsubsection{Global geometric Langlands conjectures.} Lifting the object-wise statement of Theorem \ref{heckeeigen} led Beilinson and Drinfeld to the formulation of the {\it global geometric Langlands correspondence}, which says (roughly) that 
	\begin{thm}[\cite{GLCI}, \cite{GLCII}, \cite{GLCIII}, \cite{GLCIV}, \cite{GLCV}]\label{GLC}
		There exists a canonical equivalence of derived categories
		$$\D(\Bun_{G}) \xrightarrow[\cong]{\bL} \QCoh(\LS_{\check{G}}),$$
	\end{thm}
	where $\Bun_G$ is the moduli stack of $G$-bundles and $\LS_{\check{G}}$ is the stack parameterizing {\it de Rham} local systems on $X$. 
	
	The {\it Betti} global geometric Langlands correspondence was introduced by Ben-Zvi and Nadler in \cite{BZN} as a version of Theorem \ref{GLC} in the Betti sheaf-theoretic context, which still remembers information of Theorem \ref{heckeeigen} and fits into the context of topological field theory (compared
	to Theorem \ref{GLC} which fits into conformal field theory). The Betti Langlands correspondence states (roughly):
	\begin{thm}\label{BettiGLC}
		There exists a canonical equivalence of derived categories
		$$\Shv_{\Nilp}(\Bun_G) \xrightarrow[\cong]{\bL^{\operatorname{Betti}}} \QCoh(\LS^{\operatorname{Betti}}_{\check{G}}),$$
	\end{thm}
	where $\LS_{\check{G}}^{\operatorname{Betti}}$ is the stack of Betti $\check{G}$-local systems. In \cite{GLCI} it is proved that Theorem \ref{BettiGLC} is equivalent to the de Rham Langlands correspondence, and thus the proof of  Theorem \ref{GLC} gives a proof of the Betti version as well. 
	
	\subsubsection{Quantization of global geometric Langlands conjectures.} However, both the theory of modular/automorphic forms and quantum field theory suggest that there should be a generalization of  Theorem \ref{GLC} and Theorem \ref{BettiGLC}. And indeed, both sides of the equivalences $\bL$ and $\bL^{\operatorname{Betti}}$ deform over the space of {\it levels} for the groups $G$ and $\check{G}$ respectively. A level $\kappa$ for $G$ is by definition a $W$-invariant bilinear form on the Cartan Lie algebra $\mathfrak{h}$ of $\mathfrak{g}$,  whose inverse is a similar kind of datum for $\check{\mathfrak{g}}$. For a simple reductive group $G$ the space of levels is 1-dimensional. Let us temporarily be in that situation. The main idea of Drinfeld, further studied by Stoyanovsky and Gaitsgory-Lysenko is to consider {\it twisted} sheaves/$D$-modules. 
	
	Let $\cL_{\det}$ denote the determinant line bundle on $\Bun_G$, i.e., the line bundle whose fiber at a point $\cP_G \in \Bun_G$ is $\det R\Gamma(X, \mathfrak{g}_{\cP_G})$.
	Let $\cL_{\det}$ denote the determinant line bundle on $\Bun_G$, i.e., the line bundle whose fiber at a point $\cP_G \in \Bun_G$ is $\det R\Gamma(X, \mathfrak{g}_{\cP_G})$. Then for every $c \in k$ denote by
	$\D_c(\Bun_G) $ the derived category of D-modules twisted by ${\frac{c - \check{h}}{2\check{h}}}$-th power of  $\cL_{\det}$, where $\check{h}$ is the dual Coxeter number of $G$. Let $r$ be be the maximal multiplicity of arrows in the Dynkin diagram of $G$. Then the de Rham {\it quantum} global geometric Langlands conjecture says 
	
	\begin{conj}\label{qGLC}
		There exists a canonical equivalence of derived categories
		$$\D_c(\Bun_G) \xrightarrow[\cong]{\bL_c} \D_{-\frac{1}{rc}}(\Bun_{\check{G}}).$$
	\end{conj}
	At the limit $c \rightarrow 0$ the category $\D_{-\frac{1}{rc}}(\Bun_{\check{G}})$ becomes $\QCoh(\LS_{G})$, so Conjecture \ref{qGLC} limits to Theorem \ref{GLC}. However, note that unlike Theorem \ref{GLC}, Conjecture \ref{qGLC} is symmetric.
	
	The {\it quantum} Betti global geometric Langlands equivalence in the Betti setting was introduced in \cite{BZN} and states that
	
	\begin{conj}\label{qBettiGLC}
		There exists a canonical equivalence of derived categories
		$$\Shv_{c, \Nilp}(\Bun_G) \xrightarrow[\cong]{\bL_c^{\operatorname{Betti}}} \int_X \Rep_{q}(\check{G}).$$
	\end{conj}
	Here the right-hand side is the derived category of appropriately twisted Betti sheaves with singular support in the global nilpotent cone. Let us explain the left-hand side, which is {\it factorization homology} of the category of representations of the quantum group. Note that for $q=1$ we have 
	$$\int_X \Rep(\check{G}) \cong \QCoh(\LS_{\check{G}}),$$
	so in the limit $c \rightarrow 0$ Conjecture \ref{qBettiGLC} recovers Theorem \ref{BettiGLC}.
	
	\subsubsection{Topological factorization homology.} The theory of factorization homology derives from the factorization algebras of Beilinson and Drinfeld (\cite{BD1}), and is a topological version of their theory. Factorization homology with coefficients in $n$-disk algebras are homology theories for topological manifolds satisfying a generalization of the Eilenberg–Steenrod axioms for ordinary homology; as such, it generalizes ordinary homology in a way that is only defined on $n$-manifolds and not necessarily on arbitrary topological spaces. Second, these homology theories define topological quantum field theories (\cite{CG}). An important special case is that of associative algebras where factorization
	homology over the circle recovers Hochschild homology.
	
	Let us comment on the definition of $ \int_M A$ for a topological $n$-manifold $M$ and an $\bE_M$-algebra $A$ in the sense of \cite[5.4]{HA}. We can think of $A$ as a (appropriately twisted) family of $\bE_n$-algebras $A_x$ parameterized by $x \in M$. In \cite[Chapter 5]{HA}, Lurie proposes a convenient geometric way to encode this data. He defines the {\it topological Ran space} $\Ranp(M)$ as the collection of all nonempty finite subsets of $M$. Then he defines a cosheaf $\cF_A$, whose stalk at $S\in \Ranp(M)$ is $\otimes_{s \in S} A_s$. We can view $\cF_A$ as a constructible cosheaf of $\Ranp(M)$, obtained by gluing together locally constant cosheaves along the locally closet subsets of $\Ranp(M)$. 
	
	We refer to the global sections $\cF_A(\Ranp(M))$ as the
	{\it topological factorization homology} of $M$ with coefficients in $A$. A collection of striking properties of this construction is formulated and proved in \cite[Chapter 5.5]{HA}.
	
	\subsection{Main result.} The goal of this paper is to construct the quantum Betti Langlands functor $\bL_{\kappa}^{\operatorname{Betti}}$ of Conjecture \ref{qBettiGLC}. 
	
	First, let us mention that in the non-quantum case the construction of the functor is given in \cite{GLCI}. Namely, they define the functor $$\bL^{\operatorname{coarse},\operatorname{Betti}, L}: \QCoh(\LS_{\check{G}}^{\operatorname{Betti}}) \rightarrow \Shv_{\Nilp}(\Bun_G)$$
	as follows. Recall that via geometric Satake the Hecke action defines an action of $\Rep(\check{G})$ on $\Shv_{\Nilp}(\Bun_G)$ for every point $x \in X$. In \cite{NY} Nadler and Yun showed that this action is locally constant with
	respect to the point of $X$, and hence gives a natural action of $\QCoh(\LS_{\check{G}}^{\operatorname{Betti}}) $ on $\Shv_{\Nilp}(\Bun_G)$.  Then the functor $\bL^{\operatorname{coarse},\operatorname{Betti}, L}$ is given by acting on a certain object $$\Poinc_!^{\operatorname{Vac}, \operatorname{glob}} \in \Shv_{\Nilp}(\Bun_G).$$
	Lastly, in  \cite{GLCI} it is proved that $\bL^{\operatorname{coarse},\operatorname{Betti}, L}$ admits a right adjoint $\bL^{\operatorname{coarse},\operatorname{Betti}}$.
	
	However, one of the main defects of the quantum setting is that the Hecke action is more degenerate: for instance, for irrational $c$ the category $\Shv_{c}(\Gra_G)^{L^+G}$ is equivalent to the category of vector spaces, so the Hecke action carries no information. Therefore the strategy of \cite{GLCI} does not work.
	
	 But there is a tool that still works as expected in the quantum setting and provides finer information than Hecke action: {\it Whittaker coefficients}. Thus instead, we are going to use this tool to construct $\bL_{\kappa}^{\operatorname{Betti}}$. We also embed our construction of $\bL_{\kappa}^{\operatorname{Betti}}$ into the formalism of 2-Fourier-Mukai transform of \cite{GLCV}. Recall that this is an equivalence between {\it sheaves of categories} over the 2-stacks $\ger_{Z_G}$ and $\ger_{\pi_1(\check{G})}$, which classify gerbes on $X$ with respect to $Z_G$ and $\pi_1(\check{G})$. 
	
	In Section \ref{FM} we upgrade the categories $\Shv_{\kappa, \Nilp}(\Bun_G)$ and $\int_X \Rep_{q}(\check{G})$ to sheave of categories over and $\ger_{Z_G}$ and $\ger_{\pi_1(\check{G})}$ respectively, and the assertion is that the resulting two sheaves of categories map to one another under
	the 2-categorical Fourier-Mukai transform. This implies that $\bL_{\kappa}^{\operatorname{Betti}}$ intertwines convolution action of $\Shv(\Bun_{Z_G})$ on $\Shv_{\kappa, \Nilp}(\Bun_G)$ with the action of $\QCoh(\ger_{\pi_1(\check{G})})$ on $\int_X \Rep_{q}(\check{G})$ via the 1-categorical Fourier-Mukai equivalence
	$$\Shv(\Bun_{Z_G}) \cong\QCoh(\ger_{\pi_1(\check{G})}).$$
	
	As an application, using this we generalize the result of \cite{Bwhit} from the case of adjoint group $G$ to the case of arbitrary $G$.

	\subsection{Outline of the proof.}
	
	Recall that there are several properties of $\bL^{\operatorname{Betti}}$ that fix the functor uniquely. Among these is the condition that ``the first
	Fourier coefficient"  functor on $\D({\Bun_G})$ correspond to the functor $\Gamma( \LS_{\check{G}}^{\operatorname{Betti}}, -)$. 
	
	The quantum version of this compatibility can be expressed as the commutativity of the diagram 
	\begin{equation}
		\begin{tikzcd}
			\Whit_{\kappa}(G)_{\Ranp}\ar[rr,rightarrow, "\operatorname{FLE}"']&   	&\Gamma^{\lax}(\Ranp, \Rep_{q}(\check{G}))\\
			\Shv_{\kappa, \Nilp}(\Bun_G)\ar[rr, rightarrow, "\bL_{\kappa}^{\operatorname{Betti}}"']\ar[u, rightarrow, "\coeff^{\loc}"']&   &\int_X \Rep_{q}(\check{G}).\ar[u, rightarrow, "\Gamma_q"']
		\end{tikzcd}
	\end{equation}
	Here the top arrow is the Fundamental Local Equivalence (see Subsection \ref{FLE}) between the category of Whittaker ind-constructible sheaves on the Beilinson-Drinfeld Grassmannian and (the appropriate $\Ranp$-version of) the category of representations of the quantum group, $\coeff^{\loc}$ is the Whittaker coefficient functor defined in Subsection \ref{Betticoeff}, and $\Gamma_q$ is the functor discussed in Corollary \ref{glaxvsfactho}. Let us say more about the definition of $\Gamma^{\lax}(\Ranp, \Rep_{q}(\check{G}))$. As was mentioned, from the category $\Rep_{q}(\check{G})$ we can form a cosheaf of categories $\Rep_{q}(\check{G})_{\Ranp}$ on $\Ranp(X)$. Then $\Gamma^{\lax}(\Ranp, \Rep_{q}(\check{G}))$ is the category of {\it lax global sections} of this cosheaf. Informally, this means that  for every $x \in X$ we need to give an element $V_x \in \Rep_{q}(\check{G})_x$, and for every specialization $x \rightsquigarrow y$ we need to give a map $m_{x \rightsquigarrow y}(V_x) \rightarrow V_y$, where $m_{x \rightsquigarrow y}:  \Rep_{q}(\check{G})_x \rightarrow  \Rep_{q}(\check{G})_y$ is the corresponding specialization map.
	
	So to give the construction of $\bL_{\kappa}^{\operatorname{Betti}}$ the first thing we do is, parallel to the de Rham case (\cite[Proposition 4.3.4]{GL2}), prove that $\Gamma_q$ is fully faithful (Corollary \ref{glaxvsfactho}). Then we formulate the property that distinguished $\int_X \Rep_{q}(\check{G})$ inside $\Rep_{q}(\check{G})_{\Ranp}$. Informally, the property says that the collection $$\{V_x\}_{x \in X} \in \Gamma^{\lax}(\Ranp, \Rep_{q}(\check{G}))$$ lies in $\int_X \Rep_{q}(\check{G})$ if and only if for every specialization $x \rightsquigarrow y$ the natural map 
	$$V_x \rightarrow m_{x \rightsquigarrow y}^R(V_y)$$
	is an isomorphism.
	
	Let us explain how to check this property on $X^2$. Let us also assume, for simplicity, that the tangent bundle to $X$ is trivial. 
	It suffices to show that for every $\cG \in \Rep_{q}(\check{G})_{X^2 \setminus X_{\Delta}}$, and $$\cF \in \Shv_{\kappa \boxtimes \triv, \Nilp \times T^* X^2 }(\Bun_G \times X^2)$$ we have 
	\begin{equation}\label{outline eq}
		\langle \coeff^{\loc}(\cF )_{ X_{\Delta}}, m_{ \Rep_{q}(\check{G})_{X^2 \setminus X_{\Delta}}}(\cG)\rangle \cong \langle \coeff^{\loc}(\cF )_{X^2 \setminus X_{\Delta}},  \cG \rangle,
	\end{equation}
	where $\langle -, - \rangle$ are the duality pairings, and 
	$$m_{ \Rep_{q}(\check{G})_{X^2 \setminus X_{\Delta}}}: \Rep_{q}(\check{G})_{X^2 \setminus X_{\Delta}} \rightarrow \Rep_{q}(\check{G})_{X_{\Delta}}$$
	is the map 
	$$\LocSys(X^2 \setminus X_{\Delta}) \otimes  \Rep_{q}(\check{G})^{\otimes 2} \rightarrow 	\LocSys(X_{\Delta}) \otimes \Rep_{q}(\check{G})$$
	coming from the $\bE_2$-structure on $\Rep_{q}(\check{G})$ and encoding all $m_{x \rightsquigarrow y}$.
	
	To show (\ref{outline eq}), we consider the functors 
	\begin{equation}\label{eq1}
		\Poinc_{*, \co} := ( \coeff^{\loc})^{\vee} : \Rep_{q}(\check{G})_{X^2 \setminus X_{\Delta}}\rightarrow \Shv_{\kappa \boxtimes \triv, \Nilp \times T^* (X^2 \setminus X_{\Delta})}(\Bun_G \times (X^2 \setminus X_{\Delta}))
	\end{equation}
	and 
	\begin{equation}\label{eq2}
		F: \Shv_{\triv \boxtimes \kappa, T^*(X^2 \setminus X_{\Delta} )\times \Nilp}((X^2 \setminus X_{\Delta})\times \Bun_G) \rightarrow \Shv(X^2 \setminus X_{\Delta})
	\end{equation}
	defined as $$\cF^{\prime} \rightarrow p_{{X^2 \setminus X_{\Delta}}_{\disj}, *}(\cF \stackrel{!}{\otimes } \cF^{\prime}).$$
	We show that they have the property of being {\it defined and codefined by a kernel}, a notion introduced in \cite{AGKRRV2}. In particular, this means that both (\ref{eq1}) and (\ref{eq2}) commute with $m_{ \Rep_{q}(\check{G})_{X^2 \setminus X_{\Delta}}}$, which gives the desired equivalence (\ref{outline eq}).


		\subsection{Notations and conventions.}
		
		\subsubsection{Categories.} We use the formalism of $\infty$-categories developed in \cite{LHTT}, \cite{HA}, and theory of DG categories as understood in \cite{GR}. By a DG category we mean $k$-linear stable $\infty$-category. We let $\DGCat$ be the category
		of presentable DG categories with continuous functors with monoidal structure given by the Lurie tensor product $\otimes$. 
		
		\subsubsection{Lie theory.}
		
		Throughout the paper $G$ will stand for a reductive group over $k$ of adjoint type. We choose Borel subgroup $B \subset G$, the opposite Borel $B^- \subset G$ and the Cartan subgroup $T= B \cap B^-$. Let $N$ ($N^-$) be the unipotent radical of $B$ ($B^-$). Let $\Lambda$ ($\check{\Lambda}$) denote the lattice of weights (coweights) of $G$. Let $\Lambda^+$ ($\check{\Lambda}^+$) denote the subset of dominant weights (coweights). Let $\cI$ denote the set of nodes for the Dynkin diagram of $G$. For $i \in \cI$, we let $\alpha_i$ ($\check{\alpha}_i$) denote the
		corresponding simple root (coroot). Let $2\rho$ denote the sum of simple roots and $2\check{\rho}$ the sum of simple coroots. Let $\check{G}$ be the Langlands dual group over $k$ for $G$. We denote by $LG$ ($L^+G$) the loop (arc) group of $G$.
		
		Let $\mathfrak{g}_{\irr} \subset \mathfrak{g}$ denote the reduced closed subscheme of irregular elements.
		We let $\cN \subset \mathfrak{g}$ denote the nilpotent cone. We let $\cN^{\irr} := \cN \cap \mathfrak{g}_{\irr}$ denote the subscheme of irregular
		nilpotent elements. We let $\cN^{\reg} \subset \cN$ denote the open complement to $\cN^{\irr}$, which parametrizes
		of regular nilpotent elements.
		
		\subsubsection{Levels.} A level $\kappa$ for $G$ is a $G$-invariant symmetric bilinear form 
		$$\kappa: \Sym^2(\mathfrak{g}) \rightarrow \bC.$$
		Denote by $\kappa_{\mathfrak{g}, \operatorname{crit}}$ the critical level, i.e. $-\frac{1}{2}$ times the Killing form. 
		Throughout the paper we will assume that the level $\kappa$ is non-degenerate, i.e. that $\kappa - \kappa_{\mathfrak{g}, \operatorname{crit}}$ is nondegenerate as a bilinear form. The {\it dual} level $\check{\kappa}$ on $\check{G}$ for $\kappa$ is the unique nondegenerate level such that the restriction of $\check{\kappa} - \kappa_{\check{\mathfrak{g}}, \operatorname{crit}}$ to $\mathfrak{t}^*$ and the restriction of $\kappa - \kappa_{\mathfrak{g}, \operatorname{crit}}$ to $\mathfrak{t}$ are dual symmetric bilinear forms.
		
		Suppose $G$ is simple. A level $\kappa$ is {\it rational} if $\kappa$ is a rational multiple of the Killing form and {\it irrational} otherwise. A level $\kappa$ is {\it positive} if $\kappa - \kappa_{\mathfrak{g}, \operatorname{crit}}$ is a positive rational multiple of the Killing form. A level is {\it negative} if it is not positive or critical. For general reductive $G$, we say a level $\kappa$ is rational, irrational, positive, or negative if its restrictions to each simple factor are so.
		
		Let $q$ be the {\it quantum parameter} associated to $\kappa$. Namely,
		$$q = \exp(2\pi i q_{\kappa}),$$
		where $q_{\kappa}$ is the quadratic form $\Lambda \rightarrow k$ such that the associated bilinear form is $\kappa$.
		\subsubsection{(Twisted) sheaves.}
		We will work in one of the following sheaf-theoretic contexts:
		\begin{itemize}
			\item $\D(-) $ denotes the category of D-modules;
			\item $\Shv(-)$ denoted the category of \emph{weakly constructible} Betti sheaves (as in \cite{handbook}).
		\end{itemize}
		For a level $\kappa$, let $\D_{\kappa}(\Bun_{G})$ be the category of  $\kappa$-twisted D-modules as in \cite[Section 10.1]{GLCII}.
		
		By \cite[Proposition 3.1.9]{GL} from a fixed level $\kappa$ one gets a factorization gerbe $\cG_G$ on $\Gra_G$, which descends to $\Bun_G$. We let $\Shv_{\kappa}(\Bun_G)$ be the category of sheaves on $\Bun_G$ twisted by $\cG_G$.
		\begin{rem}
		Note that by \cite[G.2]{AGKRRV} all Betti sheaves on $\Bun_G$ with microsupport in $\Nilp \subset T^*\Bun_G$ are weakly constructible.
		\end{rem}
		
		\subsubsection{Ran space(s) and factorization.}
		In the geometric situation, the notation $\Ranp$ for the Ran space is borrowed from \cite{GLCII}. We refer to \emph{loc.cit.} for generalities on Ran spaces and factorization.
		
		The topological Ran space $\Ranp(M)$ is as in \cite[Chapter 5]{HA}.
		
		\subsection{Acknowledgements.} I am very grateful to Sam Raskin for suggesting this project and for many discussions along the way. I am also grateful to Dennis Gaitsgory and Lin Chen for many useful suggestions and explanations. I thank Lin Chen for explaining the contents of Section \ref{5} to me, and Dennis Gaitsgory for his comments on an earlier version of this text.
		
		\section{Generalities.}
	
	\subsection{Whittaker coefficients and Poincare functors.}\label{Betticoeff}
	
	\subsubsection{}
	Recall the Kirillov model of the Whittaker category from \cite[Section 1.6]{Whit}. Namely, consider a category $\mathbf{C}$ with the action of $\bG_m \ltimes \bG_a$. Let $\on{Kir}(\mathbf{C})$ denote the quotient $\mathbf{C}^{\bG_m} / \mathbf{C}^{\bG_m \ltimes \bG_a}$, where we use the forgetful functor 
		$$\mathbf{C}^{\bG_m \ltimes \bG_a} \hookrightarrow \mathbf{C}^{\bG_m}.$$
		The natural map $p: \mathbf{C}^{\bG_m} \rightarrow \on{Kir}(\mathbf{C})$ admits a fully faithful right adjoint denoted by $p^R$ and a fully fithful left adjoint denoted by $p^L$.

		Consider $LN$ and its non-degenerate character $\chi$, and $\mathbf{C}=\Shv_{\kappa}(\Gra_{G, \rho(\omega_X), x})^{\ker(\chi)}$. Then $\mathbf{C}$ admits an action of $\bG_m$ via $\check{\rho}$, and we get 
		\begin{ntn}
			$\Whit_{\kappa, x}(G):=\on{Kir}(\Shv_{\kappa}(\Gra_{G, \rho(\omega_X), x})).$
		\end{ntn}
		The similar construction works in the setting of Beilinson-Drinfeld affine Grassmannian:
		\begin{ntn}
			$\Whit_{\kappa}(G):=\on{Kir}(\Shv_{\kappa}(\Gra_{G, \rho(\omega_X)})).$
		\end{ntn}

		\subsubsection{}
		The Whittaker category admits a global model. Namely, for $\cZ \rightarrow \Ranp$ let $\BunNbZ$ be the stack introduced in \cite[Section 10.2.1]{GLCIII}. Denote by $\bar{p}_{\cZ}$ the projection 
		$$\BunNbZ \rightarrow \Bun_G \times \cZ.$$
		Denote by 
		$\Shv_{\kappa}(\BunNbZ)$ the corresponding category of twisted sheaves, by means of the pullback of the gerbe corresponding to $\kappa$. Following \cite[Section 4.7]{Whit} we can define a full Kirillov model subcategory $$\Whit_{\kappa}(G)^{\operatorname{glob}} \subset \Shv_{\kappa}(\BunNbZ)^{\mathbb{G}_m}.$$
		We have a naturally defined map 
		$$\pi_{\cZ}: \Gra_{G, \rho(\omega_X), \cZ} \rightarrow \BunNbZ.$$
		
		\begin{thm}\cite[Theorem 5.1.4]{Whit}
			The functor $(\pi_{\cZ})^*$ induces an equivalence 
			$$\Whit_{\kappa}(G)^{\operatorname{glob}} \xrightarrow{\cong}\Whit_{\kappa}(G).$$
		\end{thm}
		
		We refer to \cite[Section 9]{GLCII} and \cite[Section 10.2]{GLCIII} for more details.
		
		\begin{rem}\label{dualwhit}
			Note that the pairing $\Gamma_!(- \overset{*}{\otimes}-)$ induced the duality between between $\Whit_{\kappa}(G)$ and $\Whit_{-\kappa}(G)$.
		\end{rem}
		
		\subsubsection{}
		
		We define the functor $$\coeff^{\loc}:\Shv_{\kappa, \Nilp}(\Bun_G) \rightarrow \Whit_{\kappa}(G)$$
		as $\coeff^{\loc}:= \Av_{*, \chi}\circ\pi^!$. Alternatively, the functor $\coeff^{\loc}$ can also be described as the composition
		\begin{equation}
			\Shv_{\kappa, \Nilp}(\Bun_G) \xrightarrow{!-\text{pull}} \Shv_{\kappa}(\BunNbZ)^{\bG_m} \xrightarrow{\Av_{*}}  \Shv_{\kappa}(\BunNbZ/\ker(\chi))^{\bG_m} \xrightarrow{p} \Whit_{\kappa}(G)
		\end{equation}
		We define the functor $$\coeff_!^{\loc}:\Shv_{\kappa, \Nilp}(\Bun_G) \rightarrow \Whit_{\kappa}(G)$$
		as the composition 
		\begin{equation}
			\Shv_{\kappa, \Nilp}(\Bun_G) \xrightarrow{*-\text{pull}} \Shv_{\kappa}(\BunNbZ)^{\bG_m} \xrightarrow{\Av_{!}}  \Shv_{\kappa}(\BunNbZ/\ker(\chi))^{\bG_m} \xrightarrow{p} \Whit_{\kappa}(G).
		\end{equation}
		To define Poincare functors, recall the Beilinson's projector functor $P$ of \cite[Section 18.2]{AGKRRV}. We have an adjunction 
		\[
		\begin{tikzcd}
			P:\Shv_{\kappa}(\Bun_G)\ar[r,  shift left, ""]
			& \arrow[l, hookrightarrow, shift left, ""]   \Shv_{\kappa, \Nilp}(\Bun_G): \iota.
		\end{tikzcd}
		\]
		We define functor $\Poinc_!, \Poinc_* : \Whit_{\kappa} \rightarrow \Shv_{\Nilp}(\Bun_G)$ to be the compositions 
		\begin{equation}
			\begin{split}
				\Whit_{\kappa}(G) \xrightarrow{p^L} \Shv_{\kappa}(\BunNbZ/\ker(\chi))^{\bG_m} \xrightarrow{\oblv} \Shv_{\kappa}(\BunNbZ)^{\bG_m} \\
				\Shv_{\kappa}(\BunNbZ)^{\bG_m}  \xrightarrow{ !-\text{push}}\Shv_{\kappa}(\Bun_G)\xrightarrow{P}\Shv_{\kappa, \Nilp}(\Bun_G)
			\end{split}
		\end{equation} and \begin{equation}
			\begin{split}
				\Whit_{\kappa}(G) \xrightarrow{p^R} \Shv_{\kappa}(\BunNbZ/\ker(\chi))^{\bG_m} \xrightarrow{\oblv} \Shv_{\kappa}(\BunNbZ)^{\bG_m} \\
				\Shv_{\kappa}(\BunNbZ)^{\bG_m}  \xrightarrow{ *-\text{push}}\Shv_{\kappa}(\Bun_G)\xrightarrow{P}\Shv_{\kappa, \Nilp}(\Bun_G)
			\end{split}
		\end{equation} respectively. 
		\begin{rem}
			Note that the functors $(\Poinc_!, \coeff^{\loc})$ form an adjoint pair.
		\end{rem}
		\begin{rem}\label{duall}
			The functor $\Poinc_!$ is identified with the dual of $\coeff^{\loc}_!$ under the duality between  $\Shv_{\kappa, \Nilp}(\Bun_G)$ and $\Shv_{-\kappa, \Nilp}(\Bun_G)$ induced by $\Gamma_!(- \overset{*}{\otimes}-)$ ({\cite[Theorem G.5.2]{AGKRRV}}) and the duality between $\Whit_{\kappa}(G)$ and $\Whit_{-\kappa}(G)$ in Remark \ref{dualwhit}.
		\end{rem}
		
		\subsection{Constant term and Eisenstein series functors.}
		
		\subsubsection{} Let $\kappa$ be a level for $\mathfrak{g}$, and recall that we use the same character $\kappa$
		to denote the corresponding level for $\mathfrak{m}$. 
		
		We let
		$$\operatorname{CT}^-_!:\Shv_{\crit_G+\kappa}(\Bun_G) \to \Shv_{\crit_M+\kappa}(\Bun_M)$$ 
		denote the composition
		\begin{multline*} 
			\Shv_{\kappa}(\Bun_G)
			\overset{(p^-)^*}\longrightarrow 
			\Shv_{\kappa}(\Bun_{P^-}) \simeq \Shv_{\on{dlog}(\det^{\otimes \frac{1}{2}}_{\Bun_{G,M}})+\kappa}(\Bun_{P^-})
			\overset{\otimes \det^{\otimes -\frac{1}{2}}_{\Bun_{G,M}}}\longrightarrow \\
			\to \Shv_{\kappa}(\Bun_{P^-})\overset{(q^-)_!}\longrightarrow 
			\Shv_{\kappa}(\Bun_M)=  \Shv_{\kappa}(\Bun_M)\overset{[-\on{shift}]}\longrightarrow \Shv_{\kappa}(\Bun_M),
		\end{multline*} 
		
		where $$\operatorname{shift}:=  \dim. \operatorname{rel.}(\Bun_{P^-} / \Bun_M).$$
		
		By \cite[Proposition 1.1.2]{DG} the functor $\operatorname{CT}^-_!$ admits a right adjoint $\Eis_*$ when restricted to one connected component. 
		
		Moreover, one can also define the functor $\Eis_!$ as $[\on{shift}]\circ (p^-)_! \circ (\otimes \det^{\otimes -\frac{1}{2}}_{\Bun_{G,M}})\circ (q^-)^*$. By  \cite[Proposition 1.2.3]{DG}, the right adjoint $\on{CT}^-_*$ to $\Eis_!$ exists and is canonically isomorphic to $\operatorname{CT}_!$ when restricted to one connected component.
		
		\begin{pr}\label{eisctnilp}
			The functors $\operatorname{CT}^-_!$, $\Eis_*$, $\Eis_!$, $\on{CT}^-_*$ preserve nilpotence of singular support.
		\end{pr}
		\begin{proof}
			Follows from the Hecke compatibility of the functors (\cite{BG} and \cite{FH}) and Hecke description of the nilpotent singular support of \cite{AGKRRV} as in \cite[Proposition 1.4.2]{GR}.
		\end{proof}
		\begin{ntn}
			We are going to use the same notation for the functors between $\Shv_{\kappa, \Nilp}(\Bun_G)$ and $\Shv_{\kappa, \Nilp}(\Bun_M)$ induced by $\operatorname{CT}^-_!$, $\Eis_*$, $\Eis_!$, $\on{CT}^-_*$ via Proposition \ref{eisctnilp}.
		\end{ntn}
		
		Recall the following statement from \cite{DG}:
		\begin{thm}\cite[Theorem 1.2.3]{DG}\label{secondadj}
			The functors $\operatorname{CT}^{\mu, -}_!$ and $\operatorname{CT}^{\mu}_*$ are canonically isomorphic.
		\end{thm}
		
		Note that under duality in Remark \ref{duall} we have 
		\begin{equation}
			\Eis_! \cong (\operatorname{CT}_!)^\vee;
		\end{equation}
		\begin{equation}
			\operatorname{CT}_! \cong (\Eis_!)^\vee.
		\end{equation}
		
		\subsubsection{}	For an $M$-torsor $\cP_M$ we define the translated constant term and Eisenstein functors
		$$\Eis^-_{!,\cP_{Z_M}}:\Shv_{\crit_G+\kappa}(\Bun_G)\rightleftarrows \Shv_{\crit_M+\kappa}(\Bun_M):\on{CT}_{*,\cP_{Z_M}}^-$$
		by
		$$\Eis^-_{!,\cP_{Z_M}}:=\Eis^-_!\circ (\on{transl}_{\cP_{Z_M}})_* \text{ and }
		\on{CT}_{*,\cP_{Z_M}}^-:=(\on{transl}_{\cP_{Z_M}})^*\circ \on{CT}_*^-,$$
		where $$\on{transl}_{\cP_{Z_M}}:\Bun_M\to \Bun_M$$
		is the automorphism given by translation by $\cP_M$.
		
		\subsection{Naive and corrected quantum Jacquet functors.} 
		The goal of the rest of this section is to recall the constructions of local and global versions of enhanced Jacquet functors, which are going to be used in the sequel.
		
		We now recall the contents of \cite[Section 5.1]{GL}. Let $P$ be a parabolic group of $G$, let $P \twoheadrightarrow M$ be the Levi quotient, and let $N_P$ be the corresponding unipotent group. Let $A$ denote roots of unity in $k$. 
		
		Consider
		\[
		\begin{tikzcd}
			\Gra_G & \arrow[l, "p"']  \Gra_P  \arrow[r, "q"] & \Gra_M.
		\end{tikzcd}
		\]
		
		\begin{lm}\cite[5.1]{GL}
			For any $S \rightarrow \Ranp$ the pullback along $q$ induces the equivalences
			$$\ger(S \times_{\Ranp} \Gra_M) \rightarrow \ger(S \times_{\Ranp} \Gra_P)$$
			and 
			$$\fger(S \times_{\Ranp} \Gra_M) \rightarrow \fger(S \times_{\Ranp} \Gra_P)$$
		\end{lm}
		
		For $\cG_G \in \fger(\Gra_G)$, denote its pullback to $\Gra_P$ by $\cG_P$. Let $\cG_G$ be the corresponding factorization gerbe on $\Gra_M$. Then for any $S \rightarrow \Ranp$ the functors 
		$$p^!: \Shv_{\kappa}(\Gra_G) \rightarrow \Shv_{\kappa}(\Gra_P)$$
		and
		$$q_*: \Shv_{\kappa}(\Gra_P) \rightarrow \Shv_{\kappa}(\Gra_M)$$
		are well-defined by \cite{DG}. We refer to the resulting map $ q_*\circ p^!$ of sheaves of categories over $\Ranp$ 
		$$\Shv_{\kappa}(\Gra_G) \rightarrow  \Shv_{\kappa}(\Gra_M)$$
		as the {\it naive Jacquet functor}.
		
		However, to make this functor compatible with geometric Satake and with factorization, one need to perform a certain correction by applying a cohomlogical shift. This procedure is described in \cite[5.3]{GL}.
		
		We denote the resulting functor 
		$$\Shv_{\kappa}(\Gra_G) \rightarrow  \Shv_{\kappa}(\Gra_M)$$
		by $J_{\Gra}^{-. *}$.
		However, we will also need a version of Jacquet functors which encode the Hecke structure as well, i.e., the so-called {\it enhanced} Jacquet functors. We give the definition in Section \ref{enhjacq}. Before that, we digress to introduce the background, mostly discussed in \cite{GLsmall} and \cite{GLCIII}.
		
		\subsection{The metaplectic semi-infinite category.}
		
		Consider the factorization category 
		$$I(G, P^{-}) := \Shv_{\kappa}(\Gra_G)^{LN_P^{-} \cdot L^+M}.$$
		Here the renormalization procedure is with respect to $L^+M$ and is given as in \cite[Section 1.2]{GLCIII}.
		
		As in \cite[14.2]{GLsmall}, we have that $I(G, P^{-})$ admits a natural action of the spherical category $\Sph_q(G)$. Also, as in \cite[Section 1.3]{GLCIII}, we get that $I(G, P^{-})$ admits a (corrected) action of $\Sph_q(M)$ commuting with the action $\Sph_q(G)$.
		
		\subsubsection{} We have a factorization functor
		\begin{equation}\label{SI1}
			\begin{split}
				I(G, P^{-}) := \Shv_{\kappa}(\Gra_G)^{LN_P^{-} \cdot L^+M} \rightarrow \Shv_{\kappa}(\Gra_{P^-})^{LN_P^{-} \cdot L^+M} \cong \\
				\cong \Sph_{q}(M) \xrightarrow{[\operatorname{shift}]} \Sph_{q}(M),
			\end{split}
		\end{equation}
		where 
		\begin{itemize}
			\item the second arrow is $(p^-)^!$,
			\item the last arrow is the shift by $\langle \lambda, 2 \check{\rho}_{P} \rangle$ on $\Gr_M^{\lambda}$.
		\end{itemize}
		
		Analogous to \cite[Proposition 1.5.3]{ICII} and \cite{FH}, the functor  (\ref{SI1}) admits a factorization left adjoint:
		
		\begin{equation}
			\textbf{ind}_{\Sph_{q} \rightarrow \frac{\infty}{2}}: \Sph_{q}(M) \xrightarrow{-[\operatorname{shift}]} \Sph_{q}(M) \xrightarrow{\cong} \Shv_{\kappa}(\Gra_{P^-})^{LN_P^{-} \cdot L^+M} \xrightarrow{(p^-)_!}I(G, P^{-}) .
		\end{equation}
		
		Define 
		$$\Delta^{-, \frac{\infty}{2}} := 	\textbf{ind}_{\Sph_{q} \rightarrow \frac{\infty}{2}}(\delta_{1, \Gra_M}).$$
		In other words, $\Delta^{-, \frac{\infty}{2}}$ is the $!$-extension of the dualizing sheaf on the $LN_{P^-}$-orbit through the
		origin in $\Gra_G$. In the quantum context the object $\Delta^{-, \frac{\infty}{2}}$ still satisfies the properties described in \cite[1.3.5-1.3.9]{GLCIII}.
		We also have a functor 
		\begin{equation}
			\Sph_{q}(M) \xrightarrow{-[\operatorname{shift}]} \Sph_{q}(M) \xrightarrow{(p^-)^!} \Shv_{\kappa}(\Gra_{P^-})^{LN_P^{-} \cdot L^+M} \xrightarrow{(q^-)_*}I(G, P^{-}) ,
		\end{equation}
		denoted by $	\textbf{ind}^*_{\Sph_{q} \rightarrow \frac{\infty}{2}}$. Set 
		$$\nabla_{\Sph_{q} \rightarrow \frac{\infty}{2}} := 	\textbf{ind}^*_{\Sph_{q} \rightarrow \frac{\infty}{2}}(\delta_{1, \Gra_M}).$$
		The object $\nabla_{\Sph_{q} \rightarrow \frac{\infty}{2}}$ satisfies the properties described in \cite[Section 1.3.10-1.3.12]{GLCIII}.
		Recall that there is also a factorization algebra 
		$$\ic_q^{-, \frac{\infty}{2}}$$
		in $I(G, P^{-})$. It is constructed in \cite[Section 13]{GLsmall} for the principle parabolic. The construction for general parabolic follows \cite{FH} and \cite{DL}.
		This algebra is unital and  equipped with homomorphisms of unital factorization algebras
		$\Delta_{\Sph_{q} \rightarrow \frac{\infty}{2}} \rightarrow \ic^{-, \frac{\infty}{2}}_q \rightarrow \nabla_{\Sph_{q}\rightarrow \frac{\infty}{2}}.$
		\subsection{Enhanced quantum Jacquet functors.}\label{enhjacq}
		In this section we recall the construction given in \cite[Section 15]{GLsmall} and \cite[1.8]{GLCIII} of the generalized Jacquet functors.
		The functors $J_{\Gra}^{-. !*}$, $J_{\Gra}^{-. *}$, and $J_{\Gra}^{-. !}$ are defined using the correspondence
		$$\Gra_M \leftarrow \Gra_M \times_{\Ranp} \Gra_G \rightarrow \Gra_G$$
		using the kernels $\ic_{q^{-1}}^{-, \frac{\infty}{2}}$, $\nabla_{\Sph_{q^{-1}} \rightarrow \frac{\infty}{2}}$, and $\Delta_{\Sph_{q^{-1}}\rightarrow \frac{\infty}{2}}$ respectively. We can also perform the twisting by $\rho(\omega_X)$ as in \cite[1.8.5]{GLCII}, and we denote the result by the same notation $J_{\Gra}^{-. !*}$, $J_{\Gra}^{-. *}$, and $J_{\Gra}^{-. !}$.
		\begin{rem}
			The functors $J_{\Gra}^{-, !*}$, $J_{\Gra}^{-, *}$, and $J_{\Gra}^{-, !}$, when restricted to $\Whit_{\kappa}(G)$ map to $\Whit_{\kappa}(M) \subset \Shv_{\kappa}(\Gra_{M})$. 
		\end{rem}
		\begin{ntn}
			We will denote by $J_{\Whit}^{-, !*}$, $J_{\Whit}^{-, *}$, and $J_{\Whit}^{-, !}$ the resulting functors 
			$$\Whit_{\kappa}(G) \rightarrow \Whit_{\kappa}(M).$$
		\end{ntn}
		
		\section{Miraculous duality for Betti sheaves.}
		
		In this section we develop the miraculous duality functor in the context of Betti sheaves and prove that it is an equivalence. The main resultion of this section, Theorem \ref{miraculous}, will be used in Section \ref{kerbun}.
		
		\subsection{The unit.} Let us introduce the following category (defined in the de Rham context in \cite{DrGa}):
		
		\begin{equation}
			\Shv_{\kappa}(\Bun_G)_{\co} := \colim_{U \subset \Bun_G} \Shv_{\kappa}(U),
		\end{equation}
		where the colimit is taken over the poset of open substacks such that the intersection with every connected component of $\Bun_G$ is quasi-compact, and the transition functors are given by $j_*$. 
		
		\begin{rem}
			There is a natural functor $$\Shv_{\kappa}(\Bun_G)_{\co}  \rightarrow \Shv_{\kappa}(\Bun_G).$$
			However, it fails to be an equivalence.
		\end{rem}
		\subsubsection{} Let $\cY$ be a QCA stack (see \cite{DrGa} for the terminology). Let $i: \cZ \hookrightarrow \cY$ be a closed substack, let $j: U \hookrightarrow \cY$ be the complement open. 
		
		Following \cite[Definition 2.1.6]{DrGa}, we introduce:
		\begin{df}
			The substack $i: \cZ \hookrightarrow \cY$ is called \emph{truncative} (resp. $j: U \hookrightarrow \cY$ is called \emph{co-truncative}) if $i^!$ is defined on all $\Shv(\cY)$ and admits a continuous right adjoint. Equivalently, if $j_*$ is defined on all $\Shv(U)$ and admits a continuous right adjoint.
		\end{df}
		
		\begin{df}
			A locally QCA stack $\cY$ is called \emph{truncatable} if it can be covered by open substacks which are co-truncative. 
		\end{df}
		
		\begin{pr}
			The stack $\Bun_G$ is truncatable.
		\end{pr}
		\begin{proof}
			Follows the proof of \cite[Theorem 4.1.12]{DrGa} by adapting the proofs of Propositions 2.5.2 and 2.3.4 to the Betti context.
		\end{proof}
		
		\begin{cor}\label{cotrunc}
			We have 
			\begin{equation}
				\Shv_{\kappa}(\Bun_G) \cong \lim_{U \subset \Bun_G} \Shv_{\kappa}(U),
			\end{equation}
			where the limit is taken over the poset of co-truncative open quasi-compact substacks of $\Bun_G$;
				\begin{equation}
				\Shv_{\kappa}(\Bun_G)_{\co} \cong \colim_{U \subset \Bun_G} \Shv_{\kappa}(U),
			\end{equation}
			where the limit is taken over the poset of co-truncative open quasi-compact substacks of $\Bun_G$.
		\end{cor}
		
		\begin{cnstr}
			By Corollary \ref{cotrunc}, the category $\Shv_{-\kappa}(\Bun_G)_{\co}$ is dualizable. Therefore, the category $$\Shv_{\kappa}(\Bun_G) \otimes \Shv_{-\kappa}(\Bun_G)_{\co}$$ can be identified with 
			\begin{equation}\label{limco}
				\lim_{U \subset \Bun_G} (\Shv_{\kappa}(U) \otimes \Shv_{-\kappa}(\Bun_G)_{\co}).
			\end{equation}
			Let us define an object $$\Delta_{*,co}(\omega) \in \Shv_{\kappa}(\Bun_G) \otimes \Shv_{-\kappa}(\Bun_G)_{\co}.$$ As a compatible system of objects in the formula (\ref{limco}), the object $\Delta_{*,co}(\omega)$ is $$j_{*, \co}\circ \Delta_*(\omega_U) \in \Shv_{\kappa}(U) \otimes \Shv_{-\kappa}(\Bun_G)_{\co}.$$ Here $\omega$ is the dualizing object.
		\end{cnstr}
		
			The object $\Delta_{*,co}(\omega)$ defines a functor $$\Vect \rightarrow  \Shv_{\kappa}(\Bun_G) \otimes \Shv_{-\kappa}(\Bun_G)_{\co},$$ and by duaity between $\Shv_{\kappa}(\Bun_G)$ and $\Shv_{-\kappa}(\Bun_G)$ described in Remark \ref{duall} this gives a functor
		$$\psid^{\mir}: \Shv_{-\kappa}(\Bun_G) \rightarrow \Shv_{-\kappa}(\Bun_G)_{\co}.$$

		The goal of this section is to prove the following Betti analog of the \emph{Miraculous duality} for $\Bun_G$:
		\begin{thm}\label{miraculous}
			The object $$\Delta_{*,co}(\omega) \in \Shv_{\kappa}(\Bun_G) \otimes \Shv_{-\kappa}(\Bun_G)_{\co}$$ defines a duality between $\Shv_{\kappa}(\Bun_G)$ and $\Shv_{-\kappa}(\Bun_G)_{\co}$. In other words, the functor $\psid^{\mir}$ is an equivalence.
		\end{thm}

		Define $$\Eis_{\co, *}^{\mu}: \Shv_{\kappa}(\Bun_M^{\mu})_{\co} \rightarrow \Shv_{\kappa}(\Bun_G)_{\co}$$
		as a compatible system of $\Eis_{*}^{\mu}$ functors on open co-truncative substacks. 
		Define $$\Eis_{\co, *}: = \bigoplus \Eis_{\co, *}^{\mu}.$$
		
			\begin{pr}\label{eisra}
			The functor $\Eis_{\co, *}$ admits a continuous right adjoint  $\Eis_{\co, *}^R$. 
		\end{pr}
		
		\begin{proof}
			By \cite[Proposition 2.4.3]{strange}, on every open co-truncative substack the functor $\Eis_{*}$ admits a finite filtration by functors $\Eis_{ !} \circ H^{\alpha}(-)$, where $H^{\alpha}(-):= \overset{\leftarrow}{h}_!(\overset{\rightarrow}{h^*}(-) \overset{*}{\otimes} K^{\alpha})$. We note that $\Eis_!$ admits a continuous right adjoint, and $H^{\alpha}$ does as well since $K^{\alpha}$ is ULA relative to $\overset{\rightarrow}{h}$. Therefore on every open co-truncative substack the functor $\Eis_{*}$ admits a continuous right adjoint, and thus $\Eis_{\co, *}$ admits a continuous right adjoint as well. 
		\end{proof}

		\subsection{Cuspidality in the co-category.}
		 As in \cite{DG}, let $$\Shv_{\kappa}(\Bun_G)_{\cusp} \subset \Shv_{\kappa}(\Bun_G)$$ be the intersection of kernels of $\CT_*$ for all proper parabolic subgroups. Equivalently, we have 
		 \begin{equation}
		 	\Shv_{\kappa}(\Bun_G)_{\cusp} \cong (\Shv_{\kappa}(\Bun_G)_{\Eis})^\perp,
		 \end{equation}
		 where $\Shv_{\kappa}(\Bun_G)_{\Eis} \subset \Shv_{\kappa}(\Bun_G)$ denotes the subcategory generated by $\Eis_!$ for all proper parabolic subgroups. 
		
		\begin{df}
			Define $\Shv_{\kappa}(\Bun_G)_{\co, \cusp} \subset \Shv_{\kappa}(\Bun_G)_{\co}$ as 
			$$\Shv_{\kappa}(\Bun_G)_{\co, \cusp} \cong (\Shv_{\kappa}(\Bun_G)_{\co, \Eis})^\perp,$$
			where $\Shv_{\kappa}(\Bun_G)_{\co, \Eis} \subset  \Shv_{\kappa}(\Bun_G)_{\co}$ is the subcategory generated by $\Eis_{*, \co}$ for all proper parabolic subgroups. Equivalently, $\Shv_{\kappa}(\Bun_G)_{\co, \cusp}$ is the subcategory  generated by $\Eis_{\co, *}^R.$
		\end{df}
		
		Recall from \cite[Proposition 1.4.6]{DG} the following characterization of the support of the cuspidal objects: 
		
		\begin{pr}\label{ug}
			There exists a quasi-compact open $j_G: U_G \subset \Bun_G$ such that for any $F \in \Shv_{\kappa}(\Bun_G)_{\cusp}$ the maps 
			$$j_{G, !} j^*_G F \rightarrow F \rightarrow j_{G, *}j^{*}_G F$$
			are isomorphisms.
		\end{pr}
		
		We claim the parallel statement also holds for $\Shv_{\kappa}(\Bun_G)_{\co, \cusp}$:
		\begin{pr}
			For any $F \in \Shv_{\kappa}(\Bun_G)_{\co, \cusp}$ the map 
			\begin{equation}
				F \rightarrow j_{G, *, \co} j^*_{G, \co} F
			\end{equation}
			is an equivalence.
		\end{pr}
		
		\begin{proof}
			It suffices to show that
			\begin{equation} \label{fib1}
				 (\on{Fib}(F \rightarrow j_{G, *, \co} j^*_{G, \co} F)) \in \Shv_{\kappa}(\Bun_G)_{\co, \Eis}.
			\end{equation}
		We will show that for every $\lambda \in \Lambda^+_{\bQ}$, such that $\lambda \notin \Sigma$ (for $\Sigma$ as in \cite[B.1.2]{DG}), and the Harder-Narasimhan component $$i_{\lambda}: \Bun_G^{(\lambda)} \rightarrow \Bun_G$$ and any $W \in \Shv_{\kappa}(\Bun_G^{(\lambda)})$ we have
		\begin{equation}\label{lambdaeis}
			i_{\lambda, *, \co}(W) \in \Shv_{\kappa}(\Bun_G)_{\co, \Eis}.
		\end{equation}
		Take $V$ as in \cite[Lemma B.3.2]{DG}. We claim that $\Shv_{\kappa}(\Bun_G^{(\lambda)})$ is generated by images of
		\begin{equation}\label{pq}
			(p|_{q^{-1}(V)})_*q^!: \Shv_{\kappa}(V) \rightarrow \Shv_{\kappa}(\Bun_G^{(\lambda)}).
		\end{equation}
		First, notice that by \cite[Lemma B.3.5]{DG} the map $q^!$ is an equivalence. Then let $C_{\lambda} \subset \Shv_{\kappa}(\Bun_G^{(\lambda)})$ be the subcategory generated by (\ref{pq}). Let $K \in \prescript{\perp}{}{C_{\lambda}}$. Then for any $N \in \Shv_{\kappa}(V)$ we have
		$$0 \cong \Hom(K,(p|_{q^{-1}(V)})_*q^!(N)) \cong \Hom((p|_{q^{-1}(V)})^*K,q^!(N)).$$
		Since $q^!$ is an equivalence and $(p|_{q^{-1}(V)})^*$ is conservative by surjectivity proved in \cite[Lemma B.3.2]{DG}, we get that $C_{\lambda} = \Shv_{\kappa}(\Bun_G^{(\lambda)})$. Finally, to see (\ref{lambdaeis}), notice that 
		$$i_{\lambda, *, \co}((p|_{q^{-1}(V)})_*q^!(N)) \cong \Eis_{*, \co}(i_{V, *, \co}(N)).$$
		\end{proof}
		
		\begin{pr}
			Let $F \in \Shv_{\kappa}(\Bun_G)_{\co}$ be such that there exists a quasi-compact open $j: U \subset \Bun_G$ and the map 
			$$F \rightarrow j_{*, \co} j^{*}_{\co}(F)$$
			is an equivalence. Then $F \in \Shv_{\kappa}(\Bun_G)_{\co, \cusp}$ if and only if $\CT_{*, \co}(F) = 0$ for all proper parabolic subgroups.
		\end{pr}
		
		\begin{proof}
		First, note that under these condition $F \in \Shv_{\kappa}(\Bun_G)_{\co, \cusp}$ if and ony if for every co-truncative $j_W: W \subset \Bun_M$ and $N \in \Shv_{\kappa}(W)$ we have 
		\begin{equation}\label{coCT1}
			\Hom(j^*\Eis_*j_{W, *}N, j^{*}_{\co}(F)) \cong 0
		\end{equation}
		for every proper parabolic subgroup.
		
		Now consider $\CT_{*, \co}(F)$. By \cite[Proposition 1.5.7]{strange} for a certain quasi-compact open $j_Q: Q \subset \Bun_M$ we have 
		$$\CT_{*, \co}(F) \cong j_{Q, *, \co} j_Q^*\CT_{*} j_* j^{*}_{\co}(F).$$
		Now the condition that $\CT_{*, \co}(F)$ is zero for all proper parabolics translates to 
		\begin{equation}\label{coCT2}
		\Hom(A,  j_Q^*\CT_{*} j_* j^{*}_{\co}(F)) \cong \Hom(j^* \Eis_! j_{Q, !}A,   j^{*}_{\co}(F))
		\end{equation}
		for all $A \in \Shv(Q)$.
		
		Since both $j_{Q, !}A$ and $j_{W, *}N$ are supported on finitely many connected components of $\Bun_M$, the fact that conditions (\ref{coCT1}) and (\ref{coCT2}) are equivalent follows from \cite[Proposition 2.4.3]{strange}.
		\end{proof}
		
		\begin{rem}
			Consider the quasi-compact open $U_G$ from Proposition \ref{ug}. Set $$\Shv_{\kappa}(U_G)_{\cusp} := \Shv_{\kappa}(U_G) \cap \Shv_{\kappa}(\Bun_G)_{\cusp}$$
			as subcategories of $\Shv_{\kappa}(\Bun_G)$, where $\Shv_{\kappa}(U_G)$ is viewed as such via the embedding $j_{G, *}$.
			
			Set $$\Shv_{\kappa}(U_G)_{\co, \cusp} := \Shv_{\kappa}(U_G) \cap \Shv_{\kappa}(\Bun_G)_{\co, \cusp}$$
			as subcategories of $\Shv_{\kappa}(\Bun_G)_{\co}$, where $\Shv_{\kappa}(U_G)$ is viewed as such via the embedding $j_{G, *, \co}$.
			
			Note that the identity functor on $\Shv_{\kappa}(U_G)$ identifies subcategories $\Shv_{\kappa}(U_G)_{\cusp}$ and $\Shv_{\kappa}(U_G)_{\co, \cusp}$.
		\end{rem}
		
		\begin{rem}
			As in \cite[3.2.4]{strange} we have a natural transformation 
			\begin{equation}\label{diff}
				\Id_{U_G} \rightarrow j_{G, \co}^*\psid^{\mir}j_{G, *}[-2\dim(\Bun_G) - \dim(Z_G)].
			\end{equation}
		\end{rem}
		
		Let 
		\begin{equation}
			\psid_{\operatorname{diff}}: \Shv_{\kappa}(U_G) \rightarrow \Shv_{\kappa}(U_G)
		\end{equation}
		denote the cone of (\ref{diff}).
		
		Let $Q$ be the open in $\Bun_M$ as in \cite[Proposition 1.5.7]{strange}. Recall the following (\cite[Theorem 4.3.1]{sch}, \cite[Proposition 3.2.6]{strange}):
		\begin{pr}\label{3.2.6}
			The functor $\psid_{\operatorname{diff}}$ admits a finite decreasing filtration, indexed by a
			poset, with subquotients being functors of the form
			$$\Shv_{\kappa}(U_G) \xrightarrow{j_{Q^{\mu}}^*\CT_* j_{G, *}} \Shv_{\kappa}(Q^\mu) \xrightarrow{F^{\mu, \mu^{\prime}}} \Shv_{\kappa}(Q^{\mu^{\prime}}) \xrightarrow{j_G^*\Eis_{*}^{\mu^{\prime}, -} j_{Q^{\mu^{\prime}, *}}} \Shv_{\kappa}(U_G),$$
			for a proper parabolic $P$ with Levi quotient $M$, where $\mu, \mu^{\prime} \in \pi_1(\Bun_M)$ and $F^{\mu, \mu^{\prime}}$ is some functor $\Shv_{\kappa}(Q^\mu)\rightarrow \Shv_{\kappa}(Q^{\mu^{\prime}})$.
		\end{pr}
		
		\begin{cor}
			The morphism (\ref{diff}) induces an isomorphism
			$$\Id_{\Shv_{\kappa}(U_G)_{\cusp}} \cong \psid^{\mir}|_{\Shv_{\kappa}(U_G)_{\cusp}}[-2\dim(\Bun_G) - \dim(Z_G)].$$
		\end{cor}
		
		\begin{proof}
			Follows from Proposition \ref{3.2.6} as in \cite[Corollary 3.2.2]{strange}.
		\end{proof}
		
		\begin{cor}\label{mircusp}
			The functor $\psid^{\mir}$ induced an equivalence
			$$\Shv_{\kappa}(\Bun_G)_{\cusp} \rightarrow \Shv_{\kappa}(\Bun_G)_{\co, \cusp}.$$
		\end{cor}
		
		\begin{pr}\label{3.3.5}
			The functor $\psid^{\mir}$ induced an isomorphism
			\begin{equation}
				\Hom_{\Shv_{\kappa}(\Bun_G)}(F^{\prime}, F) \rightarrow \Hom_{\Shv_{\kappa}(\Bun_G)_{\co}}(\psid^{\mir}(F^{\prime}), \psid^{\mir}(F)),
			\end{equation}
			given that $F^{\prime} \in \Shv(\Bun_G)_{\cusp}.$
		\end{pr}
		\begin{proof}
			Follows from results of the present section as in \cite[3.4]{strange}.
		\end{proof}

		\subsection{The strange functional equation.}

		\begin{pr}\label{4.1.2}
			For a parabolic $P$ and an opposite parabolic $P^-$ we have an equivalence of functors
			\begin{equation}\label{strange1}
				\Eis_{\co, *} \circ \psid^{\mir} \cong \psid^{\mir} \circ \Eis_!^-.
			\end{equation}
		\end{pr}
		
		\begin{proof}
			Both sides correspond to objects in 
			\begin{equation}
					\Shv_{-\kappa}(\Bun_M) \otimes \Shv_{\kappa}(\Bun_G)_{\co} \cong \lim_{U_M \subset \Bun_M} (\Shv_{-\kappa}(U_M) \otimes \Shv_{\kappa}(\Bun_G)_{\co}),
			\end{equation}
			where $U_M \subset \Bun_M$ are co-truncative opens. We claim that both sides of (\ref{strange1}) correspond to the system 
			\begin{equation}
				\{ (\Id \otimes j_{*, \co})\circ (q \times p)_*\circ \Delta_{*}(\omega_{q^{-1}(U_M)}) \in \Shv_{-\kappa}(U_M) \otimes \Shv_{\kappa}(\Bun_G)_{\co} \}_{U_M \subset \Bun_M}.
			\end{equation}
			Indeed, the left-hand side of (\ref{strange1}) corresponds to 
			$$(\Id \otimes \Eis_{\co, *})(\Delta_{*,co}(\omega_{\Bun_M})),$$
			and the right-hand side corresponds to 
			$$(\CT_* \otimes \Id)(\Delta_{*,co}(\omega_{\Bun_G})),$$
			and the assertion follows by base-change.
			
		\end{proof}
		
		Recall from Proposition \ref{eisra} that the functor $\Eis_{\co, *}$ admits a continuous right adjoint $\Eis_{\co, *}^R$.
		
		\begin{pr}\label{cteisra}
			For a parabolic $P$ the natural transformation of functors 
				\begin{equation}\label{strange2}
				\psid^{\mir} \circ \CT_!\rightarrow    \Eis_{\co, *}^R \circ \psid^{\mir} 
			\end{equation}
			is an equivalence.
		\end{pr}
		
		\begin{proof}
			We need to prove that the natural map 
			\begin{equation}\label{strange2obj}
				(\Eis_! \otimes \Id)(\Delta_{*,\co}(\omega_{\Bun_M})) \rightarrow (\Id \otimes  \Eis_{\co, *}^R )(\Delta_{*,\co}(\omega_{\Bun_G})) \in \Shv_{-\kappa}(\Bun_G) \otimes \Shv_{\kappa}(\Bun_M)_{\co} 
			\end{equation}
			is an equivalence. 
			
			Recall that 
			\begin{equation}
				\Shv_{-\kappa}(\Bun_G) \otimes \Shv_{\kappa}(\Bun_M)_{\co}  \cong \lim_{U_G \subset \Bun_G, U_M \subset \Bun_M} \Shv_{-\kappa}(U_G) \otimes \Shv_{\kappa}(U_M),
			\end{equation}
			where $j:U_G \subset \Bun_G$, $u:U_M \subset \Bun_M$ are quasi-compact co-truncative opens and the connecting functors are given by $(j^*, u^?)$ (here $u^?$ is the right adjoint to $u_*$).
			We claim that for every quasi-compact co-truncative opens $U_G \subset \Bun_G$, $u:U_M \subset \Bun_M$ the map corresponding to (\ref{strange2obj}), i.e.
			\begin{equation}\label{strange2obj2}
				(j^* \otimes u^?)\circ (\Eis_! \otimes \Id)(\Delta_{*,\co}(\omega_{\Bun_M})) \rightarrow (j^* \otimes u^?)\circ(\Id \otimes  \Eis_{\co, *}^R )(\Delta_{*,\co}(\omega_{\Bun_G}))
			\end{equation}
			is an equvalence. Then we can further rewrite (\ref{strange2obj2}) as 
			\begin{equation}\label{strange2obj3}
					(j^*\circ \Eis_! \otimes \Id)\circ (\Id \otimes u^?)(\Delta_{*,\co}(\omega_{\Bun_M})) \rightarrow  (\Id \otimes ( j^* \circ\Eis_{*, \co} \circ u_{*, co})^R)(\Delta_{*, U_G}(\omega)),
			\end{equation}
			which by definition of $\Delta_{*,\co}(\omega_{\Bun_M}))$ is the same as
			\begin{equation}\label{goal}
				(j^*\circ \Eis_! \circ u_! \otimes \Id)\circ (\Delta_{*}(\omega_{U_M})) \rightarrow  (\Id \otimes ( j^* \circ\Eis_{*, \co} \circ u_{*, co})^R)(\Delta_{*, U_G}(\omega)).
			\end{equation}
			
			In the D-module setting, by \cite[Theorem 4.1.2]{strange} we have 
			\begin{equation}
				\Eis_! \circ \psid_! \cong \psid_! \circ \Eis_{*, \co}^-.
			\end{equation}
			Therefore we also have 
			\begin{equation}
					(j^! \Eis_! u_!) \circ \psid_! \cong \psid_! \circ (j^? \Eis_{*, \co}^- u_{*, \co}).
			\end{equation}
			Note that both $(j^! \Eis_! u_!)$ and $(j^? \Eis_{*, \co} u_{*, \co})$ admit left adjoints. Indeed, $(j^? \Eis_{*, \co} u_{*, \co})$ admits a left adjoint by definition and $(j^! \Eis_! u_!)$ admits a left adjoint by \cite[Proposition 2.4.3]{strange} Thus we have 
			\begin{equation}
				(j^! \Eis_! u_!)^L \circ \psid_! \cong \psid_! \circ (j^? \Eis_{*, \co}^- u_{*, \co})^L,
			\end{equation}
			or, in other words, 
			\begin{equation}\label{aaa}
			(((j^? \Eis_{*, \co}^- u_{*, \co})^L)^\vee \otimes \Id)\Delta_!(k_{U_M}) \cong (\Id \otimes (j^! \Eis_! u_!)^L)\Delta_!(k_{U_G}).
			\end{equation}
			Rewriting (\ref{aaa}) we get
			\begin{equation}
				((u^! \CT_{*}^- j_{!})^R \otimes\Id)\Delta_!(k_{U_M}) \cong (\Id \otimes (j^! \Eis_! u_!)^L)\Delta_!(k_{U_G}).
			\end{equation}
			Now both sides of this equiation are holonomic D-modules, so applying Verdier duality we get
			\begin{equation}
				((j^! \Eis_! u_{!}) \otimes \Id)\Delta_*(\omega_{U_M}) \cong (\Id \otimes (j^* \Eis_* u_*)^R)\Delta_*(\omega_{U_G}).
			\end{equation}
			But this coincides with (\ref{goal}), hence we get that (\ref{goal}) is an equivalence.
		\end{proof}
		
		\subsection{The miraculous functor is an equivalence.}
		
		Now we are ready to prove Theorem \ref{miraculous}.
		\begin{proof}[Proof of Theorem \ref{miraculous}]
			We first notice that the case when $G$ is torus follows from Corollary \ref{mircusp}. We will use induction on the semi-simple rank of $G$, i.e. assume that the assertion holds for all proper Levi subgroups of $G$.
			
			By Proposition \ref{4.1.2} the image of subcategory $\Shv_{\kappa}(\Bun_G)_{\Eis}$ under $\psid^{\mir}$ generates $\Shv_{\kappa}(\Bun_G)_{\co, \Eis}$. By Corollary \ref{mircusp}, the image of the subcategory $\Shv_{\kappa}(\Bun_G)_{\cusp}$ under $\psid^{\mir}$ generated $\Shv_{\kappa}(\Bun_G)_{\co, \cusp}$. Thus we need to show that $\psid^{\mir}$ is fully faithful.
			
			Recall from Proposition \ref{3.3.5} that the functor $\psid^{\mir}$ induced an isomorphism
			\begin{equation}
				\Hom_{\Shv_{\kappa}(\Bun_G)}(F^{\prime}, F) \rightarrow \Hom_{\Shv_{\kappa}(\Bun_G)_{\co}}(\psid^{\mir}(F^{\prime}), \psid^{\mir}(F)),
			\end{equation}
			given that $F^{\prime} \in \Shv_{\kappa}(\Bun_G)_{\cusp}.$ Hence we need to show that the functor $\psid^{\mir}$ induced an isomorphism
			\begin{equation}
				\Hom_{\Shv_{\kappa}(\Bun_G)}(F^{\prime}, F) \rightarrow \Hom_{\Shv_{\kappa}(\Bun_G)_{\co}}(\psid^{\mir}(F^{\prime}), \psid^{\mir}(F)),
			\end{equation}
			given that $F^{\prime} \in \Shv_{\kappa}(\Bun_G)_{\Eis}.$ In other words, we need to show that for $F_M \in \Shv_{\kappa}(\Bun_M)$ the map 
			\begin{equation}
				\Hom_{\Shv_{\kappa}(\Bun_G)}(\Eis_!(F_M), F) \rightarrow \Hom_{\Shv_{\kappa}(\Bun_G)_{\co}}(\psid^{\mir}(\Eis_!(F_M)), \psid^{\mir}(F))
			\end{equation}
			is an isomoprhism.
		\end{proof}
		However, we have
			\[
		\begin{tikzcd}
			\Hom_{\Shv_{\kappa}(\Bun_G)}(\Eis_!(F_M), F)  \ar[r, rightarrow, ""']\ar[ddd, rightarrow, "\cong"']&   	\Hom_{\Shv_{\kappa}(\Bun_G)_{\co}}(\psid^{\mir}(\Eis_!(F_M)), \psid^{\mir}(F)) \ar[d, rightarrow, "\text{Proposition \ref{4.1.2}}"', "\cong"]  \\
			& \Hom_{\Shv_{\kappa}(\Bun_G)_{\co}}(\Eis_{*, \co}^{-}(\psid^{\mir}(F_M)), \psid^{\mir}(F)) \ar[d, rightarrow, "\text{Proposition \ref{eisra}}"', "\cong"] \\
			&  \Hom_{\Shv_{\kappa}(\Bun_M)_{\co}}(\psid^{\mir}(F_M), \Eis_{*, \co}^{-, R}(\psid^{\mir}(F)))\ar[d, rightarrow, "\text{Proposition \ref{cteisra}}"', "\cong"] \\
			\Hom_{\Shv_{\kappa}(\Bun_M)}(F_M, \CT_*F) \ar[r, rightarrow, "\cong"']&   \Hom_{\Shv_{\kappa}(\Bun_M)_{\co}}(\psid^{\mir}(F_M), \psid^{\mir}(\CT_*F)),
		\end{tikzcd}
		\]
		where the bottom horizontal arrow is an equivalence by induction. Hence the top horizontal arrow is an equivalence.
		
		\subsection{Miraculous duality for sheaves with nilpotent singular support.}
		
		Recall from \cite[Section 1.6]{AGKRRV2} the projector $P$ onto the category $\Shv_{\kappa, \Nilp}(\Bun_G)$. Similarly (since \cite[Theorem 1.3.7]{AGKRRV2} also holds for $\Shv_{\kappa}(\Bun_G)_{\co}$), by \cite[Remark 13.4.8]{AGKRRV} there exists a projector $P_\co$ onto the category $\Shv_{\kappa, \Nilp}(\Bun_G)_{\co}$ defined in \cite[2.5.8]{AGKRRV2}.
		
		\begin{pr}\label{2.2.2}
		We have canonical isomorphisms
		\begin{equation}\label{unit}
			(P \boxtimes \Id)(\Delta_{*,co}(\omega)) \cong (P \boxtimes P_{\co})(\Delta_{*,co}(\omega)) \cong (\Id \boxtimes P_{\co})(\Delta_{*,co}(\omega)).
		\end{equation}
		\end{pr}
		
		\begin{proof}
			Follows the proof of \cite[Proposition 2.2.2]{AGKRRV2} using \cite[Lemma 2.5.7]{AGKRRV2}.
		\end{proof}
		
		\begin{pr}
			The object (\ref{unit}) of $\Shv_{-\kappa, \Nilp}(\Bun_G) \otimes \Shv_{\kappa,\Nilp}(\Bun_G)_{\co}$ and the functor 
			\begin{equation}
				\Shv_{-\kappa, \Nilp}(\Bun_G) \otimes \Shv_{\kappa, \Nilp}(\Bun_G)_{\co} \rightarrow \Shv_{-\kappa}(\Bun_G) \otimes \Shv_{\kappa}(\Bun_G)_{\co} \xrightarrow{\Gamma_{*, \co}(- \stackrel{!}{\otimes} -)} \Vect
			\end{equation}
			define a datum of duality.
		\end{pr}
		
		\begin{proof}
			Follows from Theorem \ref{miraculous} and Proposition \ref{2.2.2} as in \cite[Proposition 2.3.3]{AGKRRV2}.
		\end{proof}
		
		\begin{cor}\label{mirtwo}
			The following diagram commutes:
				\[
			\begin{tikzcd}
				\Shv_{-\kappa, \Nilp}(\Bun_G) \otimes \Shv_{\kappa, \Nilp}(\Bun_G) \ar[rr,rightarrow, "\Gamma_!(- \stackrel{*}{\otimes} -)"]\ar[d, rightarrow, "\Id \otimes \psid^{\mir}"']&   	&	\Vect \\
				\Shv_{-\kappa, \Nilp}(\Bun_G) \otimes \Shv_{\kappa, \Nilp}(\Bun_G)_{\co}\ar[rru, rightarrow, "\Gamma_{*, \co}(- \stackrel{!}{\otimes} -)"'].&  & 
			\end{tikzcd}
			\]
		\end{cor}

		\section{Functors defined and codefined by a kernel.}

		\subsection{Functors (co)defined by a kernel.}
		For quasi-compact stacks $\cY_1$ and $\cY_2$, the notion of \emph{functors} $$\Shv(\cY_1) \rightarrow \Shv(\cY_2)$$
		\emph{defined and/or codefined by a kernel} was introduced in \cite[Appendix B]{AGKRRV2}. The following description from loc.cit. will be useful.
		
		\begin{lm}\cite[B.1.5]{AGKRRV2}\label{kernels}
			The functor 
			$$\mathbf{Q}: \Shv(\cY_1) \rightarrow \Shv(\cY_2)$$
			is defined by a kernel if an only if for any algebraic stack $\cZ$ the functors $\id_{\cZ} \boxtimes \mathbf{Q}$ commute with the following operations:
			\begin{itemize}
				\item for a map $f: \cZ^{\prime} \rightarrow \cZ$, the diagram
				\[
				\begin{tikzcd}
					\Shv(\cZ^{\prime}  \times \cY_1)\ar[rr,rightarrow, "\id_{\cZ^{\prime}} \boxtimes \mathbf{Q}"']\ar[d, rightarrow, "(f\times \id)_{\blacktriangle}"']&   	&	\Shv(\cZ^{\prime}  \times \cY_2)\ar[d, rightarrow, "(f\times \id)_{\blacktriangle}"']\\
					\Shv(\cZ  \times \cY_1)\ar[rr, rightarrow, "\id_{\cZ} \boxtimes \mathbf{Q}"']&  & \Shv(\cZ  \times \cY_2)
				\end{tikzcd}
				\]
				commutes;
				\item for a map $f: \cZ^{\prime} \rightarrow \cZ$, the diagram
				\[
				\begin{tikzcd}
					\Shv(\cZ^{\prime}  \times \cY_1)\ar[rr,rightarrow, "\id_{\cZ^{\prime}} \boxtimes \mathbf{Q}"']&   	&	\Shv(\cZ^{\prime}  \times \cY_2)\\
					\Shv(\cZ  \times \cY_1)\ar[rr, rightarrow, "\id_{\cZ} \boxtimes \mathbf{Q}"']\ar[u, rightarrow, "(f\times \id)^!"']&   &\Shv(\cZ  \times \cY_2)\ar[u, rightarrow, "(f\times \id)^!"']
				\end{tikzcd}
				\]
				commutes;
				
				\item for $\cF \in \Shv(\cZ)$, the diagram 
				\[
				\begin{tikzcd}
					\Shv(\cZ^{\prime}\times \cZ  \times \cY_1)\ar[rr,rightarrow, "\id_{\cZ^{\prime}\times \cZ} \boxtimes \mathbf{Q}"']&   	&	\Shv(\cZ^{\prime}  \times \cZ \times \cY_2)\\
					\Shv(\cZ  \times \cY_1)\ar[rr, rightarrow, "\id_{\cZ} \boxtimes \mathbf{Q}"']\ar[u, rightarrow, "\cF \boxtimes -"']&  & \Shv(\cZ  \times \cY_2)\ar[u, rightarrow, "\cF \boxtimes -"']
				\end{tikzcd}
				\]
				commutes.
			\end{itemize}
		\end{lm}
		Motivated by Lemma \ref{kernels}, we introduce the following definition. Let $\AGCat$ and its dual $\AGCat^{\operatorname{left}}$ be the categories introduced in \cite{GKRVnew}.
		
		\begin{df}
			For $\mathbf{C}$, $\mathbf{D} \in \AGCat$, we say that a functor between plain DG categories 
			$$f:\mathbf{C}_{k} \rightarrow \mathbf{D}_{k}$$
			is \emph{defined by a kernel} if it upgrades to a functor 
			$$F: \mathbf{C} \rightarrow \mathbf{D} \in \AGCat.$$
		\end{df}
		
		\begin{df}
			For $\mathbf{C}$, $\mathbf{D} \in \AGCat^{\operatorname{left}}$, we say that a functor between plain DG categories 
			$$f:\mathbf{C}_{k}\rightarrow \mathbf{D}_{k}$$
			is \emph{codefined by a kernel} if it upgrades to a functor 
			$$F: \mathbf{C} \rightarrow \mathbf{D} \in\AGCat^{\operatorname{left}}.$$
		\end{df}
		
		The following is a corollary of \cite[Proposition 2.7.2]{GKRVnew}:
		
		\begin{lm}\label{cod}
			A functor $$f:\mathbf{C}_{k} \rightarrow \mathbf{D}_{k}$$ defined by a kernel is also codefined by a kernel if $F$ admits a right adjoint in $\AGCat$.
		\end{lm}
		
		%
		\subsection{Poincare series functors are defined and codefined by a kernel.}
		
		The goal of this section is to prove the following:
		
		\begin{pr}\label{poinc}
			The functor 
			$$\Poinc_!: \Whit_{\kappa}(G)\rightarrow \Shv_{\kappa, \Nilp}(\Bun_G) \in \DGCat$$
			is defined and codefined by a kernel.
		\end{pr}
		
		\begin{ntn}
			Fix an an affine scheme $S$. Let $$\Shv_{\kappa\boxtimes \triv, \Nilp \times T^*S}(\Bun_G \times S)_{\Eis, !} \subset \Shv_{\kappa\boxtimes \triv, \Nilp \times T^*S}(\Bun_G \times S)$$ denote the subcategory generated by essential images of functors $$\Eis_!: \Shv_{\kappa, \Nilp}(\Bun_M) \rightarrow  \Shv_{\kappa, \Nilp}(\Bun_G)$$ for all proper parabolics. Set 
			$$\Shv_{\kappa\boxtimes \triv, \Nilp \times T^*S}(\Bun_G \times S)_{\cusp} :=  (\Shv_{\kappa\boxtimes \triv,\Nilp \times T^*S}(\Bun_G \times S)_{\Eis, !})^{\perp}.$$
		\end{ntn}
		
		We have a recollement
		\[
		\begin{tikzcd}
			\Shv_{\kappa\boxtimes \triv, \Nilp \times T^*S}(\Bun_G \times S)_{\Eis, !} \ar[r, hookrightarrow,  shift left, "\iota_{\Eis, !}"]
			& \arrow[l,  shift left, "p_{\Eis, !}"]  \Shv_{\kappa\boxtimes \triv, \Nilp \times T^*S}(\Bun_G \times S)\ar[r, rightarrow,  shift left, "p_{\cusp}"] & \arrow[l, hookrightarrow, shift left, "\iota_{\cusp}"]   \\
			\ar[r, rightarrow,  shift left, "p_{\cusp}"]& \arrow[l, hookrightarrow, shift left, "\iota_{\cusp}"]\Shv_{\kappa\boxtimes \triv, \Nilp \times T^*S}(\Bun_G \times S)_{\cusp}.& 
		\end{tikzcd}
		\]
		
		\begin{lm}\label{p_eis}
			The endofunctors 
			$\iota_{\Eis, !} \circ p_{\Eis, !}$ and $
			\iota_{\cusp}\circ p_{\cusp}$
			are defined and codefined by a kernel.
		\end{lm}
		
		\begin{proof}
			Note that by construction the functor $\CT_*$ is defined by a kernel. On the other hand, by the Second Adjointness Theorem (\cite[Theorem 1.2.3]{DG}), the fact that $\Eis_*$ is defined by a kernel, and Lemma \ref{cod} we get that $\CT_*$ is also codefined by a kernel.

			Note that the functor $\Eis_!$ is codefined by a kernel by construction. However, by \cite[Theorem 0.1.8]{strange} and the fact that  $\CT_*$ is defined by a kernel, we get that $\Eis_!$ is also defined by a kernel. The assertion of the lemma follows. 
		\end{proof}
		
		Recall the construction of the quantum constant term functor from \cite[Section 8.1.8]{GLCII}. For an algebraic group $K$ denote by $\tau_K$ the Chevalley involution. Let $\delta_K:= \dim \Bun_K$. Recall the translation maps $\operatorname{transl}$ from \cite[Section 1.7]{GLCIII}.
		\begin{pr}\label{poincct}
			For $\cZ \rightarrow \Ranp$, the following diagram of functors commutes:
			
			\begin{equation}\label{poincctdiag}
				\begin{tikzcd}
					\Whit_{\kappa}(G)_{\cZ}\ar[d, rightarrow, "\Poinc_{G, !, \cZ}{[d]}"']\ar[rrr, rightarrow, "\operatorname{ins.unit}_{\cZ}"']&   &	&	\Whit_{\kappa}(G)_{\cZ^{\subseteq}} \ar[d, rightarrow, "J^{-, !}_{\Whit, \tau}"']  \\
					\Shv_{\kappa \boxtimes \triv,\Nilp \times T^*\cZ}(\Bun_G \times \cZ)\ar[d, rightarrow, "\CT_{*, \rho_P(\omega_X), \cZ}"']&   & &\Whit_{\kappa}(M)_{\cZ^{\subseteq}} \ar[d, rightarrow, "\Poinc_{M, !, \cZ}"']\\
					\Shv_{\kappa_M \boxtimes \triv, \Nilp \times T^*\cZ}(\Bun_M \times \cZ)\ar[rrr, leftarrow, "(\on{pr}_{\on{small},\cZ})_!"]&   & 	&\Shv_{\kappa_M \boxtimes \triv, \Nilp \times T^*\cZ}(\Bun_M \times \cZ^{\subseteq}),
				\end{tikzcd}
			\end{equation}
			where 
			
			\begin{itemize}
				
				\item $\on{CT}_{*,\rho_P(\omega_X)}\simeq (\on{transl}_{2\rho_P(\omega_X)})^*\circ \tau_M\circ \on{CT}^-_{*,\rho_P(\omega_X)}\circ \tau_G \simeq$  $$\simeq
				\tau_M\circ (\on{transl}_{-2\rho_P(\omega_X)})^*\circ \on{CT}^-_{*,\rho_P(\omega_X)}\circ \tau_G;$$

				\item $J^{-,!}_{\Whit,\tau}=\tau_M\circ J^{-,!}_{\Whit}\circ \tau_G$;

				\item $d=-2\delta_{(N_P)_{\rho_P(\omega_X)}}-\delta_{(N^-_P)_{\rho_P(\omega_X)}}+2\delta_{N^-_P}$. 
				
			\end{itemize}
		\end{pr}
		
		\begin{rem}
			Proposition \ref{poincct} is a generalization of \cite[Corollary 10.1.8]{GLCIII}, where the statement is for $\kappa = \kappa_{\operatorname{crit}}$, and the proof relies on that condition.
		\end{rem}
		
		In what follows, we refer to \cite[Section 10.5]{GLCIII} for the notation, definitions, and properties related to Zastava spaces.
		\begin{ntn}
			Denote by $\Poinc_!^{\on{all}}$ and $\Poinc_*^{\on{all}}$ the compositions $\pi_! \circ \oblv$ and $\pi_* \circ \oblv$ respectively.
		\end{ntn}
		
		\begin{proof}
			Note that by Proposition \ref{eisctnilp} we have an equivalence of functors $$P \circ \CT_{*, \rho_P(\omega_X), \cZ} \cong \CT_{*, \rho_P(\omega_X), \cZ} \circ P: \Shv_{\kappa\boxtimes \triv}(\Bun_G \times \cZ) \rightarrow \Shv_{\kappa\boxtimes \triv, \Nilp \times T^*S}(\Bun_M \times \cZ).$$
			Hence it suffices to prove that the diagram 
			\begin{equation}
				\begin{tikzcd}
					\Whit_{\kappa}(G)_{\cZ}\ar[d, rightarrow, "\Poinc_{G, !, \cZ}^{\on{all}}{[d]}"']\ar[rrr, rightarrow, "\operatorname{ins.unit}_{\cZ}"']&   &	&	\Whit_{\kappa}(G)_{\cZ^{\subseteq}} \ar[d, rightarrow, "J^{-, !}_{\Whit, \tau}"']  \\
					\Shv_{\kappa \boxtimes \triv}(\Bun_G \times \cZ)\ar[d, rightarrow, "\CT_{*, \rho_P(\omega_X), \cZ}"']&   & &\Whit_{\kappa}(M)_{\cZ^{\subseteq}} \ar[d, rightarrow, "\Poinc_{M, !, \cZ}^{\on{all}}"']\\
					\Shv_{\kappa_M \boxtimes \triv}(\Bun_M \times \cZ)\ar[rrr, leftarrow, "(\on{pr}_{\on{small},\cZ})_!"]&   & 	&\Shv_{\kappa_M \boxtimes \triv}(\Bun_M \times \cZ^{\subseteq}),
				\end{tikzcd}
			\end{equation}
			By \cite[Proposition 10.6.8]{GLCIII} the functor
			$$\Whit_{\kappa, !}(G)_\cZ \overset{\on{ins.unit}_\cZ}\longrightarrow \Whit_{\kappa, !}(G)_{\cZ^\subseteq} \overset{J^{-,!}_{\Whit}}\longrightarrow  \Whit_{\kappa_M, !}(M)_{\cZ^\subseteq}$$
			corresponds to 
			\[
			\begin{split}
				&(\on{transl}_{\rho_P(\omega_X)})^*\circ (\fs_{\cZ^{\subseteq}})_*\circ \sff_{M,\cZ}^! \circ \\
				&\left(\left(({}'\wt\sfp^-_\cZ)^!(-)\right) \sotimes ({}'\ol\sfp_\cZ)^!\circ (j_\cZ)_!(\omega_{\Bun_{P^-}\times \cZ})\right)[-\on{shift}+2\delta_{(N_P)_{\rho_P(\omega_X)}}+\delta_{(N^-_P)_{\rho_P(\omega_X)}}],
			\end{split}
			\]
			where $$ \on{shift}:= \operatorname{dim.rel.(\Bun_{P^{-}} / \Bun_M)},$$ and the maps are as in the diagram 
			\begin{equation} \label{e:CT co of Poinc diag 3}
				\vcenter
				{\xy
					(40,20)*+{\Bun_M\times \cZ^\subseteq}="X";
					(60,0)*+{\Bun_M\times \cZ.}="X'";
					(20,40)*+{\BunNbMZsub}="Z";
					(-40,20)*+{\BunNbZ}="U";
					(-20,40)*+{\Zas_\cZ}="V";
					(0,60)*+{\Zas^{\Ranp}_{\cZ}}="S";
					(-60,40)*+{\Gra_{G,\rho(\omega_X),\cZ}}="T";
					(20,0)*+{\BunPmtpZ}="R";
					(-10,0)*+{\Bun_{P^-}\times \cZ}="P";
					{\ar@{->}^{\ol\sfp_{M,\cZ^\subseteq}} "Z";"X"};
					{\ar@{->}^{\fs_{\cZ^\subseteq}} "S";"Z"};
					{\ar@{->}_{'\wt\sfp^-_\cZ} "V";"U"};
					{\ar@{->}_{\pi_\cZ} "T";"U"};
					{\ar@{->}_{\sff_{M,\cZ}} "S";"V"};
					{\ar@{^{(}->}_j "P";"R"};  
					{\ar@{->}^{'\ol\sfp_\cZ} "V";"R"};  
					{\ar@{->}^{\wt\sfq^-_\cZ} "R";"X'"};  
					{\ar@{->}^{\on{id}\times \on{pr}_{\on{small},\cZ}} "X";"X'"};  
					\endxy}
			\end{equation} 	
			
			Therefore the composition 
			\begin{equation}
				(\Id \times \on{pr}_{\on{small},\cZ})_! \circ \Poinc_{M, !, \cZ}^{\on{all}} \circ J^{-, !}_{\Whit, \tau} \circ \operatorname{ins.unit}_{\cZ}
			\end{equation}
			identifies with
			\begin{equation}\label{poinccteq1}
				\begin{split}
					&(\Id \times \on{pr}_{\on{small},\cZ} )_!\circ (\ol\sfp_{M,\cZ^\subseteq})_!\circ (\on{transl}_{\rho_P(\omega_X)})^*\circ (\fs_{\cZ^{\subseteq}})_*\circ \sff_{M,\cZ}^! \circ \\
					& \left(\left(({}'\wt\sfp^-_\cZ)^!(-)\right) \sotimes ({}'\ol\sfp_\cZ)^!\circ (j_\cZ)_!(\omega_{\Bun_{P^-}\times \cZ})\right) [-\on{shift}+2\delta_{(N_P)_{\rho_P(\omega_X)}}+\delta_{(N^-_P)_{\rho_P(\omega_X)}}].
				\end{split}
			\end{equation}
			Note that by \cite[Lemma 10.3.2]{GLCIII} the map $\sff_{M,\cZ}$ is pseudo-proper and universally homologically acylic, and also the map $\fs_{\cZ^{\subseteq}}$ is pseudo-proper, thus we can rewrite (\ref{poinccteq1}) as 
			\begin{equation}\label{poinccteq2}
				(\wt\sfq^-_\cZ)_! \circ ('\ol\sfp_\cZ)_!	\circ
				\left(\left(({}'\wt\sfp^-_\cZ)^!(-)\right) \sotimes ({}'\ol\sfp_\cZ)^!\circ (j_\cZ)_!(\omega_{\Bun_{P^-}\times \cZ})\right)[-\on{shift}+2\delta_{(N_P)_{\rho_P(\omega_X)}}+\delta_{(N^-_P)_{\rho_P(\omega_X)}}].
			\end{equation}
			Moreover, by \cite[Lemma 4.1.10]{Lin} we rewrite (\ref{poinccteq2}) as 
			\begin{equation}\label{poinccteq3}
				(\wt\sfq^-_\cZ)_! \circ ('\ol\sfp_\cZ)_!	\circ (j_{\Zas_{\cZ}})_! \circ (j_{\Zas_{\cZ}})^!
				\circ ({}'\wt\sfp^-_\cZ)^!(-)[-\on{shift}+2\delta_{(N_P)_{\rho_P(\omega_X)}}+\delta_{(N^-_P)_{\rho_P(\omega_X)}}],
			\end{equation}
			where the maps are as in the diagram 
			\begin{equation}
				\vcenter
				{\xy
					(60,0)*+{\Bun_M\times \cZ.}="X'";
					(-40,20)*+{\BunNbZ}="U";
					(-20,40)*+{\Zas_\cZ}="V";
					(-20,60)*+{\stackrel{\circ}{\Zas}_{_\cZ}}="S";
					(-60,40)*+{\Gra_{G,\rho(\omega_X),\cZ}}="T";
					(20,0)*+{\BunPmtpZ}="R";
					(-10,0)*+{\Bun_{P^-}\times \cZ}="P";
					{\ar@{->}_{'\wt\sfp^-_\cZ} "V";"U"};
					{\ar@{->}_{\pi_\cZ} "T";"U"};
					{\ar@{->}_{j_{\Zas_{\cZ}}} "S";"V"};
					{\ar@{^{(}->}_j "P";"R"};  
					{\ar@{->}^{'\ol\sfp_\cZ} "V";"R"};  
					{\ar@{->}^{\wt\sfq^-_\cZ} "R";"X'"};  
					
					\endxy}
			\end{equation} 	
			
			But the composition $$'\ol\sfp_\cZ \circ j_{\Zas_{\cZ}}$$ is smooth, so (\ref{poinccteq3}) identifies with 
			
			\begin{equation}\label{poinccteq4}
				\begin{split}
					(\wt\sfq^-_\cZ)_! \circ ('\ol\sfp_\cZ)_!	\circ (j_{\Zas_{\cZ}})_! \circ (j_{\Zas_{\cZ}})^*
					\circ ({}'\wt\sfp^-_\cZ)^*(-)[2(\dim({\stackrel{\circ}{\Zas}_{_\cZ}}) - \dim(\BunPmtpZ) )] \cong \\
					\cong 	(\wt\sfq^-_\cZ)_! \circ (({}'\wt\sfp^-_\cZ)^*(-)\stackrel{*}{\otimes} ({}'\ol\sfp^-_\cZ)^*(j_! \underline{k}))
				\end{split}
			\end{equation}
			by \cite[Theorem 4.1.10]{Lin}.
			Finally, right-hand side of (\ref{poinccteq4}) identifies with 
			$$\CT_{*, \rho_P(\omega_X), \cZ} \circ \Poinc_{G, !, \cZ}^{\on{all}}{[\operatorname{shift}]}$$
			by \cite[Proposition 4.1.8]{Lin}.
		\end{proof}
		\begin{proof}[Proof of Proposition \ref{poinc}]
			By Lemma \ref{p_eis} it suffices to check that the compositions 
			\begin{equation}
				\iota_{\cusp} \circ p_{\cusp}\circ \Poinc_!
			\end{equation}
			and 
			\begin{equation}\label{poinceq1}
				\iota_{\Eis, !} \circ p_{\Eis, !}\circ \Poinc_!
			\end{equation}
			are defined by a kernel.
			
			By \cite[Section 8.5]{GR} the cone of 
			$$\Poinc_!^{\on{all}} \rightarrow \Poinc_*^{\on{all}}$$ lies in the full subcategory generated by 
			$$\Eis_!: \Shv_{\kappa}(\Bun_M) \rightarrow  \Shv_{\kappa}(\Bun_G)$$
			for all proper parabolics of $G$. Therefore, by Propositon \ref{eisctnilp}, the cone of $$\Poinc_! \rightarrow \Poinc_*$$ lies in the full subcategory generated by $$\Eis_!: \Shv_{\kappa, \Nilp}(\Bun_M) \rightarrow  \Shv_{\kappa, \Nilp}(\Bun_G).$$
			Hence we can identify 
			$$\iota_{\cusp} \circ p_{\cusp}\circ \Poinc_! \cong \iota_{\cusp} \circ p_{\cusp}\circ \Poinc_*,$$
			and the latter composition is defined by a kernel by definition of $\Poinc_*$.
			
			To check that (\ref{poinceq1}) is defined by a kernel, it suffices to show that 
			$\CT_* \circ \Poinc_!$ is defined by a kernel, which follows from Proposition \ref{poincct} by induction.
		\end{proof}
		\begin{cor}\label{poincco}
			The functor $(\coeff^{\loc}_!)^{\vee} \cong \Poinc_!$ is defined and codefined by a kernel. 
		\end{cor}
		\subsection{Kernels given by objects from \texorpdfstring{$\Shv_{\Nilp}(\Bun_G)$}{Shv_Nilp(Bun_G)}.}\label{kerbun}
		
		We get the all Betti sheaves version of \cite[Corollary 4.2.8]{AGKRRV2}:
		\begin{thm}\label{nilpdefined}
			For any $$\cF \in \Shv_{\triv \boxtimes -\kappa, T^*X^k_{\disj} \times \Nilp}(X^k_{\disj} \times \Bun_G)$$ the functor 
			$$F: \Shv_{\triv \boxtimes \kappa, T^*X^k_{\disj} \times \Nilp}(X^k_{\disj} \times \Bun_G) \rightarrow \Shv(X^k_{\disj})$$
			defined as $$\cF^{\prime} \rightarrow p_{X^k_{\disj}, !}(\cF \stackrel{*}{\otimes } \cF^{\prime})$$
			is defined and codefined by a kernel.
		\end{thm}
		
		\begin{proof}
			Follows from Corollary \ref{mirtwo}.
		\end{proof}

		\section{Constructible cosheaves of categories and lax global sections.}\label{5}
		
		\subsection{Constructible cosheaves of categories.}
		
		\begin{ntn}\cite[Definition 2.12]{Lej}
			Denote by $\ST$ the category of conically stratified $D_{\omega}$ topological spaces.
		\end{ntn}
		
		\begin{ntn}\cite[Definition 1.22]{Lej}
			For a topological space $X$ let $\Shv(X, \mathbf{D})$ (resp. $\cshv(X, \mathbf{D})$)  be the category of $\mathbf{D}$-valued hyper(co)complete (co)sheaves on $X$. 
			For $$X \rightarrow A \in \ST$$ and any symmetic monoidal category $\mathbf{D}$, let $\Shv_A(X, \mathbf{D})$ (resp. $\cshv_A(X, \mathbf{D})$) be the category of $\mathbf{D}$-valued hyper(co)complete $A$-constructible (co)sheaves on $X$. 
		\end{ntn}
		
		\begin{ex}
			Informally, a cosheaf $\cF \in \cshv(X, \mathbf{D})$ gives the data of an element $\cF_x \in \mathbf{D}$ for every $x \in X$, and for every specialization map $x \rightsquigarrow y$ the data of a map $\cF_x \rightarrow \cF_y$.
		\end{ex}
		
		For a map $f: (X \rightarrow A) \rightarrow (Y \rightarrow B) \in \ST$ we have a pair of adjoint functors
		\begin{equation}
			\begin{tikzcd}
				f_*: \cshv_A(X, \mathbf{D}) \ar[r, rightarrow,  shift right, ""]
				& \arrow[l,  shift right, ""]  \cshv_B(Y, \mathbf{D}) : f^*,
			\end{tikzcd}
		\end{equation}
		where $f_*$ stands for the cosheaf pushforward, and $f^*$ stands for cosheaf hyperpullback (\cite[Definition 1.5, Corollary 2.18]{Lej}).
		
		\begin{rem}
			The category $\cshv(X, \mathbf{D})$ has a canonical symmetric monoidal structure. The subcategory $\cshv_A(X, \mathbf{D})$ is stable under this tensor product.
		\end{rem}
		
		Let $\Exit_A$ be the exit-path category of \cite[Appendix A.6]{HA}. Let $\Enter_A:= \Exit_A^{\op}$.
		
		\begin{rem}
			A symmetic monoidal functor $\mathbf{D}_1 \rightarrow \mathbf{D}_2$ induces a symmetric monoidal functor 
			$$\cshv_A(X, \mathbf{D}_1) \rightarrow \cshv_A(X, \mathbf{D}_2).$$
		\end{rem}
		
		\begin{thm}\cite[Corollary 3.12]{Lej}\label{shvexit}
			For $X\rightarrow A \in \ST$ we have 
			$$\operatorname{Fun}(\Exit_A, \mathbf{D}) \cong \Shv_A(X, \mathbf{D}),$$
			$$\operatorname{Fun}(\Enter_A, \mathbf{D}) \cong \cshv_A(X, \mathbf{D}).$$
			
			For a map of $A$-stratified spaces $f:X \rightarrow Y$ we have an obvious functor 
			\begin{equation}\label{enterpath}
				f: \Enter_A(X) \rightarrow \Enter_A(X),
			\end{equation}
			such that the diagram
			\[
			\begin{tikzcd}
				\operatorname{Fun}(\Enter_A(Y), \mathbf{D}) \ar[r, rightarrow, "f^*"']\ar[d, rightarrow, "\cong"']&   	\operatorname{Fun}(\Enter_A(X)), \mathbf{D}) \ar[d, rightarrow, "\cong"']  \\
				\cshv_A(Y, \mathbf{D})\ar[r, rightarrow, "f^*"']&   \cshv_A(X, \mathbf{D}).
			\end{tikzcd}
			\]
			is commutative. Also, the diagram 
			\[
			\begin{tikzcd}
				\operatorname{Fun}(\Enter_A(X), \mathbf{D}) \ar[r, rightarrow, "f_*"']\ar[d, rightarrow, "\cong"']&   	\operatorname{Fun}(\Enter_A(Y)), \mathbf{D}) \ar[d, rightarrow, "\cong"']  \\
				\cshv_A(X, \mathbf{D})\ar[r, rightarrow, "f_*"']&   \cshv_A(y, \mathbf{D}).
			\end{tikzcd}
			\]
			is commutative, where the top arrow is given by the left Kan extension along (\ref{enterpath}).
		\end{thm}
		\begin{ex}
			When $A$ is trivial, we have $$\Enter_A \cong X^{\operatorname{Spc}},$$ and 
			$$\cshv_A(X, \mathbf{D}) \cong \lim_{X^{\operatorname{Spc}}} \mathbf{D} =: \LS(X, \mathbf{D}).$$
		\end{ex}
		
		\subsection{Lax global sections of a constructible cosheaf of categories.}
		\begin{ntn}
			Denote by $\underline{\DGCat}$ the $(\infty, 2)$-category of DG categories.  
		\end{ntn}
		
		\begin{df}
			For an $(\infty, 2)$-category $\mathbf{D}$, define 
			\begin{equation}\label{cshvlaxnat}
				\cshv_A(X, \mathbf{D})^{\laxnat}:= \mathbf{Fun}(\Enter_A, \mathbf{D})^{\laxnat},
			\end{equation}
			where the right-hand side is constructed in \cite[Construction 4.5]{cospan}. This is the $(\infty, 2)$-category of constructible cosheaves valued in $\mathbf{D}$ and lax natural transformations.
		\end{df}
		
		\begin{ex}
			Informally, an arrow $$\cF_1 \rightarrow \cF_2 \in \cshv_A(X, \mathbf{D})^{\laxnat}$$ is the data of a natural transformation in $\mathbf{D}$
			\[
			\begin{tikzcd}
				\cF_{1, x} \ar[r, rightarrow, ""']\ar[d, rightarrow, ""']&   	\cF_{1, y} \ar[d, rightarrow, ""']  \\
				\cF_{2, x}\ar[ru, Rightarrow, ""']\ar[r, rightarrow, ""']&   \cF_{2, y}
			\end{tikzcd}
			\]
			for every specialization $x \rightsquigarrow y$.
		\end{ex}
		\begin{ntn}
			Denote by $$\Vect_X \in \cshv_A(X, \underline{\DGCat})^{\laxnat}:= \mathbf{Fun}(\Enter_A, \underline{\DGCat})^{\laxnat}$$ the constant functor with values $\Vect \in \underline{\DGCat}$.
		\end{ntn}
		
		\begin{thm}\cite[Theorem 3.5.4]{AGH}\label{laxnatgroth}
			For any scaled simplicial set $S$ there is an equivalence of $(\infty, 2)$-categories
			$$2\operatorname{-coCart}_{/ S} \cong \mathbf{Fun}(S, \operatorname{BiCat}_{\infty})^{\laxnat},$$
			where $2\operatorname{-coCart}_{/ S}$ stands for an $(\infty, 2)$-category of 2-coCartesian fibrations and maps of scaled simplicial sets over $S$.
		\end{thm}
		
		\begin{lm}\label{end}
			We have an equivalence of categories $$\End_{ \cshv_A(X, \underline{\DGCat})^{\laxnat}}(\Vect_X) \cong \cshv_A(X, \Vect).$$
		\end{lm}
		
		\begin{proof}
			By Theorem \ref{laxnatgroth}, we have $$\End_{ \cshv_A(X, \underline{\DGCat})^{\laxnat}}(\Vect_X)  \cong \Fun_{\Fun(\Enter_A, \operatorname{BiCat}_{\infty})^{\laxnat}}(\ast_X, \Vect_X) \cong \Fun(\Enter_A, \Vect),$$
			where $\ast$ is the point space. But by Theorem \ref{shvexit} \[\Fun(\Enter_A, \Vect) \cong  \cshv_A(X, \Vect).\]
		\end{proof}
		
		\begin{rem}
			By Verdier duality (\cite[Theorem 5.5.5.1]{HA}) we have $$ \cshv_A(X, \Vect) \cong \Shv_A(X, \Vect).$$
		\end{rem}
		
		\begin{df}
			For an object $\cF \in \cshv_A(X, \underline{\DGCat})^{\laxnat}$, define its \emph{lax global sections} as 
			$$\Gamma^{\lax}(X, \cF) := \Fun_{\cshv_A(X,\underline{\DGCat})^{\laxnat}} (\Vect_X, \cF).$$
			
		\end{df}
		
		\begin{rem}
			By Lemma \ref{end}, $\Gamma^{\lax}(X, \cF)$ is acted on by $ \cshv_A(X, \Vect) \cong \Shv_A(X, \Vect).$
		\end{rem}
		
		\begin{rem}\label{sect}
			Let $\tilde{\cF}$ be the coCartesian fibration corresponding to $$\cF \in \cshv_A(X, \underline{\DGCat})^{\laxnat}$$ by Proposition \ref{laxnatgroth}. Then 
			$$\Gamma^{\lax}(X, \cF)  \cong \operatorname{Sect}_{/ \Enter_A}(\tilde{\cF}).$$
		\end{rem}
		
		\subsection{Comparison of \texorpdfstring{$\Gamma^{\lax}(X,  -)$}{Gamma^lax(X, -)} with \texorpdfstring{$\Theta_{X, A}$}{Theta_(X, A)}.}
		
		Let us first recall the construction of 
		$$\Theta_{X, A}: \cshv_A(X, \DGCat) \rightarrow \cshv_A(X, \Vect)\Mmod(\DGCat).$$
		from \cite[6.2]{CF}.
		
		\begin{ntn}
			Let $\DGCat_{\Obj}$ denote the pith of the $(\infty, 2)$-category $\underline{\DGCat}_{\Vect /}$.
		\end{ntn}
		
		The coCartesian fibration 
		\begin{equation}
			\DGCat_{\Obj}\rightarrow 	\DGCat
		\end{equation}
		induces a coCartesian fibration 
		\begin{equation}\label{fib}
			\Fun(\Enter_A, \DGCat_{\Obj}) \rightarrow \Fun(\Enter_A, \DGCat).
		\end{equation}
		The latter corresponds to a functor 
		\begin{equation}\label{Theta}
			\Fun(\Enter_A, \DGCat) \rightarrow \operatorname{Cat}.
		\end{equation}
		Since under Grothendieck construction symmetric monoidal coCartesian fibrations correspond to lax symmetric monoidal functors (\cite[Theorem 2.1]{ramzi}), the functor (\ref{Theta}) upgrades to 
		\begin{equation}\label{upTheta}
			\Theta_{X, A}:\Fun(\Enter_A, \DGCat) \rightarrow \cshv_A(X, \Vect)\Mmod(\DGCat).
		\end{equation}
		
		\begin{lm}
			The functor $\Gamma^{\lax}(X, -)$ precomposed with $$\cshv_A(X, \DGCat) \rightarrow \cshv_A(X, \underline{\DGCat})^{\laxnat}$$ consides with $\Theta_{X, A}$.
		\end{lm}
		
		\begin{proof}
			The functor (\ref{fib}) identifies with 
			\begin{equation}
				\mathbf{Fun}(\Enter_A, \underline{\DGCat})_{\Vect_X /} \rightarrow \mathbf{Fun}(\Enter_A, \underline{\DGCat}).
			\end{equation}
			Therefore the functor (\ref{Theta}) is described as 
			$$\cF \mapsto  \operatorname{Sect}_{/ \Enter_A}(\tilde{\cF}),$$
			where $\tilde{\cF}$ be the coCartesian fibration corresponding to $\cF$. This coincides with the description in Remark \ref{sect}. The comparison directly upgrades to $\cshv_A(X, \Vect)$-module categories.
		\end{proof}
		
		\begin{cnstr}\label{dia}
			Construct a diagram $$\operatorname{fSet} \rightarrow \DGCat$$ as follows. For $I \in \operatorname{fSet}$, the corresponding DG category is $\Gamma(M^I_{\operatorname{disj}}, \cF)$.
			Here
			\begin{itemize}
				\item $\operatorname{fSet}$ is the category of finite sets with surjective maps,
				\item $M^I_{\operatorname{disj}} \subset M^I$ the open subspace consisting of $I$-tuples of disjoint points.
			\end{itemize}
			Let us now describe the connecting functors
			$$\Gamma(M^I_{\operatorname{disj}}, \cF) \rightarrow \Gamma(M^J_{\operatorname{disj}}, \cF)$$
			for $\alpha: I \twoheadrightarrow J \in \operatorname{fSet}$. 
			Note that 
			\[
			\begin{tikzcd}
				M^J \ar[r, hookrightarrow, ""']&   	M^I\\
				M^J_{\operatorname{disj}}\ar[u, hookrightarrow, ""']&   M^I_{\operatorname{disj}},\ar[u, hookrightarrow, ""']
			\end{tikzcd}
			\]
			and choose a tubular neighborhood $U_{\alpha}$ of $ M^J_{\operatorname{disj}}$ in $M^I$. 
			Then the connecting functor is defined as the composition
			\[
			\begin{tikzcd}
				\Gamma(M^I_{\operatorname{disj}}, \cF) \ar[r, rightarrow, ""]\ar[d, rightarrow, "\phi_{1}"']&   	\Gamma(M^J_{\operatorname{disj}}, \cF),  \\
				\Gamma(U_{\alpha} \cap M^I_{\operatorname{disj}}, \cF)\ar[ru, rightarrow, "\phi_{2}"']&   
			\end{tikzcd}
			\]
			where $\phi_1$ and $\phi_2$ are described as follows. The functor $\phi_1$ is the natural restriction coming from the fact that $\cF|_{M^I_{\disj}}$ is locally constant and thus also has a sheaf structure. The functor $\phi_2$ is the natural induction using the cosheaf structure of $\cF$.
		\end{cnstr}
		
		\begin{ex}
			Let $M$ be such that its tangent bundle is trivial. Note that in that case
			\begin{equation}
				U_{\alpha} \cap M^I_{\operatorname{disj}} \cong M^J_{\operatorname{disj}} \times \prod_{j \in J} \Conf_{\alpha^{-1}(j)}.
			\end{equation}
			Then for $\cF = \mathbf{A}_{\Ranp}$ with $ \mathbf{A} \in \Alg_{\bE_M}$ the composition $\phi_2 \circ \phi_1$ coincides with the composition 
			\begin{equation}
				\LocSys(M^I_{\operatorname{disj}}) \otimes \mathbf{A}^{\otimes I} \rightarrow \LocSys(U_{\alpha} \cap M^I_{\operatorname{disj}}) \otimes \mathbf{A}^{\otimes I} \rightarrow 	\LocSys(M^J_{\operatorname{disj}}) \otimes \mathbf{A}^{\otimes J},
			\end{equation}
			where the first map is the natural restriction, and the second map comes from the $\bE_M$-structure on $\mathbf{A}$.
		\end{ex}
		
		\begin{cor}\cite[Proposition 6.2.6]{CF}\label{strata}
			For $\cF \in \cshv_A(X, \DGCat)$, we have 
			$$\Gamma^{\lax}(\Ranp(M), \cF) \cong \laxlim_{I \in \operatorname{fSet}}\Gamma(M^I_{\operatorname{disj}}, \cF),$$
			where the lax limit is taken over the diagram constructed in Construction \ref{dia}
		\end{cor}
		
		\subsection{Comparison with factortization homology.}\label{glaxvsfactho}
		
		Let $\mathbf{C}$ be a symmetric monoidal category. Assume tensor product on $C$ preserves colimits separately in each variable. Let $M$ be a $k$-dimensional manifold. 
		
		\begin{thm}\cite[Theorem 5.5.4.10, Proposition 5.4.5.15]{HA}
			The category of $\bE_M$-algebra objects in $\mathbf{C}$ is equivalent to the category of $\mathbf{C}$-valued constructible factorization cosheaves on $\Ranp(M)$.
		\end{thm}
		
		For $A \in \Alg_{\bE_M}(\mathbf{C})$ denote by $A_{\Ranp}$ the corresponding cosheaf on $\Ranp(M)$. Recall the following description of of factorization (chiral) homology of $A$.
		
		\begin{thm}\cite[Theorem 5.5.4.14]{HA}
			Suppose that $M$ is connected and $A$ is locally constant. Then we have a canonical equivalence in $\mathbf{C}$:
			$$\int_M A \cong \Gamma(\Ranp(M), A_{\Ranp}).$$
		\end{thm}
		
		\begin{rem}
			For $\mathbf{C} = \DGCat$, we have 
			$$\Gamma(\Ranp(M), A_{\Ranp}) \cong \Fun_{\cshv_A(X,\underline{\DGCat})} (\mathbf{1}_X, A_{\Ranp}) \cong  \operatorname{Sect}^{\operatorname{coCart}}_{/ \Enter_A}(\tilde{A}_{\Ranp}) $$
		\end{rem}
		
		\begin{cor}\cite[Construction 4.5, Proposition 4.6]{cospan}\label{ff}
			For $\mathbf{C} = \DGCat$, a connected $M$ and a locally constant $A$, we have a natural fully faithful embedding
			\begin{equation}\label{factsubcat}
				\int_M A \hookrightarrow \Gamma^{\lax}(\Ranp(M), A_{\Ranp}).
			\end{equation}
		\end{cor} 
		
		\begin{rem}\label{factho}
			In the language of Corollary \ref{strata}, the subcategory (\ref{factsubcat})
			is described as 
			$$\{ a_I\} \in \Gamma^{\lax}(\Ranp(M), \cF) \cong \laxlim_{I \in \operatorname{fSet}}\Gamma(M^I_{\operatorname{disj}}, \cF),$$
			such that for $\alpha: I \twoheadrightarrow J \in \operatorname{fSet}$ the map 
			$$\phi_1 (a_I )\rightarrow \phi_2^R (a_J)$$
			is an isomorphism.
		\end{rem}
		The following criterion for determining whether an object of $\Gamma^{\lax}(\Ranp(M), A_{\Ranp})$ lies inside the factorization homology subcategory will be useful.
		
		\begin{pr}\label{criterion}
			Let $A \in \Alg_{\bE_M}(\DGCat)$ be locally constant and $M$ is connected. Assume in addition that $A$ is rigid. Then for $a \in  \Gamma^{\lax}(\Ranp(M), A_{\Ranp})$, it lies in the subcategory (\ref{factsubcat}) if and only if for every $I \in \operatorname{fSet}$ and $$b \in \lim_{I \twoheadrightarrow K \in \operatorname{fSet}}\Gamma(M^K_{\operatorname{disj}}, ( A^{\vee})_{\Ranp}) \subset \laxlim_{I \twoheadrightarrow K \in \operatorname{fSet}}\Gamma(M^K_{\operatorname{disj}},  A^{\vee})_{\Ranp})$$ we have 
			\begin{equation}\label{criterionformula}
				\langle a_I, \phi_{2, A^{\vee}} \circ \phi_{1, A^{\vee}}( b_I)\rangle \cong \langle \phi_{1}(a_J),  \phi_{1, A^{\vee}} (b_J) \rangle
			\end{equation}
			for every $\alpha: I \twoheadrightarrow J \in \operatorname{fSet}$, where $\langle - , - \rangle$ is the duality pairing.
		\end{pr}
		
		\begin{proof}
			By duality, criterion (\ref{criterionformula}) is equivalent to
			$$	\langle (\phi_{2, A^{\vee}})^{\vee}(a_I),  \phi_{1, A^{\vee}}( b_I)\rangle \cong \langle \phi_{1}(a_J),  \phi_{1, A^{\vee}} (b_J) \rangle ,$$
			but $(\phi_{2, A^{\vee}})^{\vee} \cong \phi_2^R$, and we obtain the result by Remark \ref{factho}.
		\end{proof}

		\section{Betti quantum Lanlgands functor via Whittaker coefficients.}\label{functor}
		
		\subsection{Construction of the local Whittaker coefficients functor.}\label{sswhit}
		
		The goal of this section is to construct the functor 
		\begin{equation}
			\coeff_!^{\loc}:	\Shv_{\kappa, \Nilp}(\Bun_G) \rightarrow \Shv_{\kappa, A_{\Ranp}} (\Gra_{G, \rho(\omega_X)})^{LN, \chi},
		\end{equation}
		where $A_{\Ranp}$ is the stratification on $\Gra_{G, \rho(\omega_X)}$ coming from stratification on $\Ranp$ by  number of distinct points and $LN$-orbits.
		
		\begin{pr}
			Functor $\coeff_!^\loc$ factors through the embedding
			$$\Shv_{\kappa, A_{\Ranp}} (\Gra_{G, \rho(\omega_X)})^{LN, \chi} \subset  \Shv_{\kappa}(\Gra_{G, \rho(\omega_X)})^{LN, \chi}.$$
		\end{pr}
		\begin{proof}
			It suffices to show that for $$\cF \in \Shv_{\kappa \boxtimes \triv, \Nilp \times T^*X(a)}(\Bun_G \times X(a))$$ and $a \in A_{\Ranp}$, the sheaf 
			$$\coeff_!^{\loc}(\cF) _{X(a)}\in \Whit_{\kappa}(G)_{X(a)} \subset \Shv_{\kappa}(\Gra_{G, \rho(\omega_X), {X(a)}})$$ is ULA with respect to the projection 
			$p_a: \Gra_{G, \rho(\omega_X), {X(a)}} \rightarrow X(a).$
			Denote by
			$$\langle -, - \rangle : \Whit_{\kappa}(G)_{X(a)} \otimes_{\Shv(X(a))} (\Whit_{\kappa}(G)_{X(a)})^\vee \rightarrow \Shv(X(a))$$
			the natural pairing induced from the pairing
			$$\Shv_{\kappa}(\Gra_{G, \rho(\omega_X), {X(a)}}) \otimes_{\Shv(X(a))} (\Whit_{\kappa}(G)_{X(a)})^\vee \rightarrow \Shv(X(a))$$
			given by 
			$$(\cF, \cG) \mapsto p_{a, *}(\cF \stackrel{!}{\otimes} \cG).$$
			Let us identify $\Whit_{-\kappa}(G)_{X(a)}$ with $(\Whit_{\kappa}(G)_{X(a)})^\vee$ as in \cite{walg}. 
			
			Note that to check that $\coeff_!^{\loc}(\cF) _{X(a)}$ is ULA with respect to $p_a$ it suffices to show that the functor 
			$$\langle \coeff_!^{\loc}(\cF) _{X(a)}, - \rangle : \Whit_{\kappa}(\Gra_{G, \rho(\omega_X), {X(a)}}) \rightarrow \Shv(X(a))$$
			sends ULA objects with respect to $p_a$ to ULA object with respect to $\Id$. Indeed, in that case the functor 
			$$ \langle \coeff_!^{\loc}(\cF) _{X(a)}, - \rangle  \otimes \Id : \Whit_{-\kappa}(G)_{X(a)} \otimes_{\Shv(X(a))} \Whit_{\kappa}(G)_{X(a)}  \rightarrow  \Whit_{\kappa}(G)_{X(a)}$$
			send ULA objects to ULA objects. However, the unit $$u \in  \Whit_{-\kappa}(G)_{X(a)} \otimes_{\Shv(X(a))} \Whit_{\kappa}(G)_{X(a)}\subset \Shv_{(-\kappa)\boxtimes \kappa}(\Gra_{G, \rho(\omega_X), {X(a)}} \times_{X(a)} \Gra_{G, \rho(\omega_X), {X(a)}})$$
			is ULA with respect to the projection to $X(a)$, and we have 
			$$ \langle \coeff_!^{\loc}(\cF) _{X(a)}, u \rangle  \otimes \Id \cong \coeff^{\loc}(\cF) _{X(a)}.$$
			Fix $\cG \in \Whit(G)_{X(a)}$ ULA with respect to $p_a$. Note that by duality 
			$$\langle \coeff_!^{\loc}(\cF) _{X(a)}, \cG \rangle \cong \langle \cF, \Poinc_{!}(\cG) \rangle:= \pi_{X(a), !}( \cF \stackrel{*}{\otimes} \Poinc_{!}(\cG)),$$
			where $\pi_{X(a)} : \Bun_G \times X(a) \rightarrow X(a)$ is the projection.
			By Corollary \ref{poincco} and \cite[Corollary B.10.7]{AGKRRV2} the sheaf $ \Poinc_{!}(\cG)$ is ULA with respect to $\pi_{X(a)}$. But since $\cF$ has nilpotent singular support, by Theorem \ref{nilpdefined} and \cite[Corollary B.10.7]{AGKRRV2} we get that 
			$$ \pi_{X(a), !}( \cF \stackrel{*}{\otimes} \Poinc_{!}(\cG))$$
			is ULA with respect to $\Id$, which finishes the proof.
		\end{proof}
		
		\subsection{Quantum Fundamental Local Equivalence.}\label{FLE}
		
		One of the key ingredients in the construction of the quantum Langlands functor will be the quantum version of the Fundamental Local Equivalence (\cite[Theorem 6.1.4]{GLCII}).
		A recent result of Gaitsgory and Hayash provides the construction of the functor in our setting:
		
		\begin{thm}\cite{GH}
			There exists a canonical functor 
			\begin{equation}\label{fle}
				\operatorname{FLE}:\Shv_{\kappa, A_{\Ranp}} (\Gra_{G, \rho(\omega_X)})^{LN, \chi} \rightarrow \Gamma^{\lax}(\Ranp, \Rep_{q}(\check{G})).
			\end{equation}
		\end{thm}
		
		We will need a stronger expected result:
		
		\begin{conj}[Gaitsgory, Lurie]\label{conjfle}
			The functor (\ref{fle}) is an equivalence.
		\end{conj}
		
		\begin{rem}
			Let us describe explicitly what Conjecture \ref{conjfle} says. Recall from \cite[Lemma 6.2.1]{CF} (and Verdier duality \cite[Theorem 5.5.5.1]{HA}) that $\Shv_{\kappa, A_{\Ranp}} (\Gra_G)$ is equivalent to the lax limit of $\LS_{\kappa}(\Gra_{G, X(a)})$ for $a \in A_{\Ranp}$, where the connecting functors are $j_b^! j_{a, !}$ for $a \rightarrow b \in A_{\Ranp}$. Moreover, the functors $j_b^! j_{a, !}$ send the subcategory $$\Shv_{\kappa, A_{\Ranp}} (\Gra_G)^{LN, \chi}_{X(a)}$$ to the subcategory $$\Shv_{\kappa, A_{\Ranp}} (\Gra_G)^{LN, \chi}_{X(b)}.$$ In other words, the category $\Whit_{\kappa}$ can also we described as a lax limit with connecting functors $j_b^! j_{a, !}$. 
			
			On the other hand, recall the description of $\Gamma^{\lax}(\Ranp, \Rep_{q}(\check{G}))$ from Corollary \ref{strata}. In \cite{GH}, the authors prove that for every $a \in A_{\Ranp}$ the functor
			$$\Shv_{\kappa, A_{\Ranp}} (\Gra_G)^{LN, \chi}_{X(a)}\rightarrow \Gamma(X(a), \Rep_{q}(\check{G}))$$
			is an equivalence. Thus, Conjecture \ref{conjfle} amounts to show that the map
			\begin{equation}
				\operatorname{FLE}(j_b^! j_{a, !}) \rightarrow \phi_{2, \Rep_{q}(\check{G})} \circ \phi_{1, \Rep_{q}(\check{G})}
			\end{equation}
			is an equivalence.
		\end{rem}
		\begin{rem}\label{reductiontot}
			The construction of the functor (\ref{fle}) from \cite{GH} (and the proof that the functor is equivalence at a point) follows the strategy of \cite{GLsmall}. Namely, it reduces to the know case of $G=T$ as follows. Gaitsgory and Hayash prove that 
			\begin{equation}
				\Gamma^{\lax}(\Ranp, \Rep_{q}(\check{G})) \cong \Gamma^{\lax}(\Ranp, \Omega_q \operatorname{-mod}^{\cE_2}(A_q\Mmod(\Rep_{\check{q}}(T)))),
			\end{equation}
			where the factorization algebra $\Omega_q$ is as in \cite{GLsmall}, and $A_q:= \cO_{N_H}$ is an algebra in $$\Rep(T_H) \subset \Rep_{\check{q}}(\check{T}).$$
			Similarly, they prove that
			\begin{equation}
				\Shv_{\kappa, A_{\Ranp}} (\Gra_G)^{LN, \chi}\cong \Omega^{\Whit} \operatorname{-FactMod}(A^{\Whit}\Mmod(\Shv_{\kappa, A_{\Ranp}} (\Gra_T))).
			\end{equation}
			Finally, they check that the factorization algebras $\Omega_q$ and $\Omega^{\Whit}$, and $A_q$ and  $A^{\Whit}$ correspond to one another under the toric FLE.
		\end{rem}
		
		\subsection{Construction of the quantum Langlands functor in the Betti context.}
		Recall that in section \ref{sswhit} we constructed
		
		\begin{equation}
			\coeff_!^\loc: \Shv_{\kappa, \Nilp}(\Bun_G) \rightarrow \Shv_{\kappa, A_{\Ranp}} (\Gra_{G, \rho(\omega_X)})^{LN, \chi}.
		\end{equation}
		
		We are now ready to prove our main result. 
		\begin{thm}\label{mainthm2}
			The composition 
			$$\operatorname{FLE} \circ \coeff^{\loc}:  \Shv_{\kappa, \Nilp}(\Bun_G) \rightarrow \Shv_{\kappa, A_{\Ranp}} (\Gra_{G, \rho(\omega_X)})^{LN, \chi} \cong \Gamma^{\lax}(\Ranp, \Rep_{q}(\check{G}))$$
			factors through $\int_X \Rep_{q}(\check{G})$.
		\end{thm}
		
		\begin{ntn}
			Denote the resulting functor 
			$$ \Shv_{\kappa, \Nilp}(\Bun_G) \rightarrow \int_X \Rep_{q}(\check{G})$$
			by $\bL$.
		\end{ntn}
		The result will follows from Propoosition \ref{criterion} and the following claim:
		\begin{pr}
			In the notation of Construction \ref{dia}, for every $\alpha: I \twoheadrightarrow J$, $\cG \in \Whit_{-\kappa, X^I}$, and $\cF \in \Shv_{\kappa \boxtimes \triv, \Nilp \times T^* X^I }(\Bun_G \times X^I)$ we have 
			\begin{equation}
				\langle \coeff_!(\cF )_{X^J_{\disj}}, \phi_{2, A^{\vee}}\circ \phi_{1, A^{\vee}}(\cG)\rangle \cong \langle \phi_{1}(\coeff_!(\cF )_{X^I_{\disj}}),  \phi_{1, A^{\vee}} (\cG) \rangle
			\end{equation}
		\end{pr}
		\begin{proof}
			Rewrite by duality 
			\begin{equation}\label{propfacthocoeff}
				\langle \coeff_!(\cF )_{X^J_{\disj}}, \phi_{2, A^{\vee}}\circ \phi_{1, A^{\vee}}(\cG)\rangle \cong \langle (\cF )_{X^J_{\disj}},  \Poinc_{!} \circ \phi_{2, A^{\vee}}\circ \phi_{1, A^{\vee}}(\cG)\rangle.
			\end{equation}
			By Corollary \ref{poincco} we can rewrite (\ref{propfacthocoeff}) as 
			$$
			\langle (\cF )_{X^J_{\disj}},  \phi_{2, A^{\vee}}\circ \phi_{1, A^{\vee}}  \circ \Poinc_{!}(\cG)\rangle,
			$$
			and by Theorem \ref{nilpdefined} we further rewrite it as 
			\begin{equation}\label{propfacthocoeff2}
				\Gamma_{c} \circ \phi_{2, \Vect}\circ \phi_{1, \Vect}  \langle (\cF )_{X^I_{\disj}},   \Poinc_{!}(\cG)\rangle_{X^I},
			\end{equation}
			where $\langle -, - \rangle_{X^I} :=  \pi_{X^I, !}(- \stackrel{*}{\otimes} -).$ However, (\ref{propfacthocoeff2}) agrees with $\langle \phi_{1}(\coeff_!(\cF )_{X^I_{\disj}}),  \phi_{1, A^{\vee}} (\cG) \rangle.$
		\end{proof}
		
		\section{2-Fourier-Mukai transform and the quantum Langlands functor.}\label{FM}
		
		\subsection{Sheaves associated with automorphic and spectral categories.}
		Let $G$ be a semisimple group. Consider categories 
		\begin{equation}
			\Shv_{\kappa, \Nilp}(\Bun_G) \quad\text{and}\quad \int_X \Rep_{q}(\check{G}).
		\end{equation}
		We will now upgrade them to the elements 
		\begin{equation}
			\underline{\Shv_{\kappa, \Nilp}(\Bun_G)} \in \operatorname{ShvCat}(\ger_{Z_G}(X)) \quad\text{and}\quad \underline{\int_X  \Rep_{q}(\check{G})} \in \operatorname{ShvCat}(\ger_{\pi_1(\check{G})}(X)).
		\end{equation}
		
		\subsubsection{} The short exact sequence of groups 
		\begin{equation}
			1 \rightarrow Z_G \rightarrow G \rightarrow G_{\ad} \rightarrow 1
		\end{equation}
		induces a map 
		\begin{equation}\label{bunge}
			\Bun_{G_{\ad}} \rightarrow \ger_{Z_G}(X).
		\end{equation}
		Since $\ger_{Z_G}(X)$ is discrete, the map (\ref{bunge}) induces on $\Shv_{\kappa,\Nilp}(\Bun_{^{\ad}})$ the structure of a sheaf of categories over $\ger_{Z_G}(X)$ denoted by $	\underline{\Shv_{\kappa, \Nilp}(\Bun_G)} $. 
		
		\subsubsection{}\label{repq(G)}
		
		Let us first upgrade the category $\Rep_q(\check{G})$ to $$\underline{ \Rep_{q}(\check{G})} \in \operatorname{ShvCat}(B^2{\pi_1(\check{G})}).$$
		
		Note that since $B^2 \pi_1(\check{G})$ has only one equivalence class of points we have 
		\begin{equation}
			\operatorname{ShvCat}(B^2{\pi_1(\check{G})}) \cong \QCoh(B {\pi_1(\check{G}))}_{\operatorname{conv}} \Mmod \cong \Rep(\pi_1(\check{G}))_{\operatorname{conv}}\Mmod.
		\end{equation}
		Here the first equivalence is given by the fiber at the trivial element of $B^2{\pi_1(\check{G})}$. Further, Fourier-Mukai transform gives 
		\begin{equation}
			\Rep(\pi_1(\check{G}))_{\operatorname{conv}}\Mmod \cong \QCoh(Z_G)\Mmod \cong \operatorname{ShvCat}(Z_G).
		\end{equation}
		Moreover, we have $Z_G \cong \pi_1(\check{G})^{\vee},$ so $$\operatorname{ShvCat}(Z_G) \cong \operatorname{ShvCat}(\pi_1(\check{G})^{\vee}).$$
		Let us now define an element in $\operatorname{ShvCat}(\pi_1(\check{G})^{\vee})$ corresponding to $$\underline{ \Rep_{q}(\check{G})} \in \operatorname{ShvCat}(B^2{\pi_1(\check{G})}).$$
		Note that $\pi_1(\check{G}) \subseteq Z_{\check{G}^{\sic}}$, and thus
		$$\operatorname{char}(Z_{\check{G}^{\sic}})\twoheadrightarrow \pi_1(\check{G})^{\vee}.$$
		Then consider the category $\Rep_{q}(\check{G}^\sic)$. For every $V \in \Rep_{q}(\check{G}^\sic)$ we have a canonical decomposition 
		$$V \cong \oplus_{\chi \in \pi_1(\check{G})^{\vee}} V_{\chi},$$
		which upgrades $\Rep_{q}(\check{G}^\sic)$ to an element in $\operatorname{ShvCat}(\pi_1(\check{G})^{\vee})$.
		\subsubsection{}
		Now let us define
		$$\underline{\int_X \Rep_{q}(\check{G})} \in \operatorname{ShvCat}(\ger_{\pi_1(\check{G})}(X)).$$
		For every affine $S$ with a map 
		$$s: S \rightarrow \ger_{\pi_1(\check{G})}(X)$$ we set
		$$s^*(\underline{\int_X \Rep_{q}(\check{G})} ) := \int_X (\Rep_q(\check{G}^{\sic}))_{\cH_s},$$
		where $\cH_s$ is the corresponding analytic $\pi_1(\check{G})$-gerbe and $(\Rep_q(\check{G}^{\sic}))_{\cH_s}$ is the $\cE_{X^{\operatorname{an}}}$-category obtained by twisting $\Rep_q(\check{G}^{\sic})$.
		
		\subsection{2-Fourier-Mukai transform and \texorpdfstring{$\bL$}{L}.} 
		
		The goal of this section is to prove:
		\begin{thm}\label{lfm}
			There exists a natural functor
			\begin{equation}\label{functorenh}
				\underline{\Shv_{\kappa, \Nilp}(\Bun_G)}  \rightarrow \operatorname{2-FM}(\underline{\int_X \Rep_{q}(\check{G})}) \in \operatorname{ShvCat}(\ger_{Z_G}(X))
			\end{equation}
			compatible with $\bL$.
		\end{thm}
		
		\begin{rem}\label{lfm2}
			The statement of Theorem \ref{lfm} is equivalent to saying that there exists a natural functor
			\begin{equation}\label{functorenh2}
				\operatorname{2-FM}(\underline{\Shv_{\kappa, \Nilp}(\Bun_G)} ) \rightarrow \underline{\int_X \Rep_{q}(\check{G})} \in \operatorname{ShvCat}(\on{Ge}_{\pi_1(\check{G})}(X))
			\end{equation}
			compatible with $\bL$.
		\end{rem}
		
		Note that Fourier-Mukai equivalence induces
		$$\operatorname{FM}: \Shv(\Bun_{Z_G}) \cong \QCoh(\ger_{\pi_1(\check{G})}(X)).$$
		Taking the fiber of (\ref{functorenh}) at the trivial gerbe $\sigma_0 \in \ger_{Z_G}(X)$ we get 
		\begin{cor}\label{equivariance}
			For $\cF \in \Shv_{\kappa, \Nilp}(\Bun_G)$ and $\cM \in \Shv(\Bun_{Z_G})$ we have
			$$\bL(\cM \textasteriskcentered{} \cF) \cong \operatorname{FM}(\cM) \ast \bL(\cF).$$
		\end{cor}

		\begin{proof}[Proof of Theorem \ref{lfm}]
			Take $\phi \in \on{Ge}_{Z_G}(X)$. Let us construct 
			\begin{equation}\label{twist}
				(\underline{\Shv_{\kappa, \Nilp}(\Bun_G)} )_{\phi} \rightarrow \operatorname{2-FM}(\underline{\int_X \Rep_{q}(\check{G})})_{\phi} \cong (\underline{\int_X \Rep_{q}(\check{G})})_{\cG_{\phi}}.
			\end{equation}
			Here $\cG_{\phi}$ is the gerbe on $\on{Ge}_{\pi_1(\check{G})}(X)$ corresponding to $\phi$, and the last equivalence follows as in \cite[(8.11)]{GLCV}.
			
			Choose a point $x \in X$. We can choose a trivialization of $\phi$ on $X \setminus x$, so we can assume that $\phi$ came from an object $$\phi \in \on{Fib}( \on{Ge}_{Z_G}(X) \rightarrow \on{Ge}_{Z_G}(X \setminus x) ) \cong Z_G(-1).$$
			We claim that the proof of Theorem \ref{mainthm2} applies in the context of (\ref{twist}) as well. Let us derive the $\phi$-analog of Conjecture \ref{FLE}. We replace $\Ranp$ by its pointed version $\Ranp_x$, and work over $\underline{x} \in \Ranp_x$, i.e. $\underline{x} :=\underline{x_1} \sqcup {x}$. Replace $$\Shv_{\kappa, A_{\Ranp}} (\Gra_{G, \rho(\omega_X)})^{LN, \chi}_{\underline{x}}:= \Shv_{\kappa, A_{\underline{x_1}}} (\Gra_{G, \rho(\omega_X)})^{LN, \chi}_{\underline{x_1}} \otimes \Shv_{\kappa, A_{x}} (\Gra_{G, \rho(\omega_X)})^{LN, \chi}$$ by 
			$$\Shv_{\kappa, A_{\Ranp_x}} (\Gra_{G, \rho(\omega_X)})^{LN, \chi}_{\underline{x}, \phi}:= \Shv_{\kappa, A_{\underline{x_1}}} (\Gra_{G, \rho(\omega_X)})^{LN, \chi}_{\underline{x_1}} \otimes \Shv_{\kappa, A_{x}} (\Gra_{G, \rho(\omega_X)})^{LN, \chi}_{\phi},$$
			where $\Shv_{\kappa, A_{x}} (\Gra_{G, \rho(\omega_X)})^{LN, \chi}_{\phi}$ comes from the decomposition coming from connected components of the affine Grassmannian:
			\begin{equation}
				\Shv_{\kappa, A_{x}} (\Gra_{G_{\ad}, \rho(\omega_X)})^{LN, \chi} \cong \oplus_{\phi \in Z_G(-1)} \Shv_{\kappa, A_{x}} (\Gra_{G, \rho(\omega_X)})^{LN, \chi}_{\phi}.
			\end{equation}
			Let $\eta_{\phi}$ be the character of $\pi_1(\check{G})$ corresponding to $\phi$. Replace the factorization cosheaf $\Rep_q(\check{G})$ on $\Ranp_x$ with 
			$$\Rep_q(\check{G})_{\underline{x}} := \Rep_q(\check{G})_{\underline{x_1}} \otimes  \Rep_q(\check{G})$$
			by the factorization cosheaf $\Rep_q(\check{G})_{\eta_{\phi}}$ on $\Ranp_x$ with
			$$\Rep_q(\check{G})_{\eta_{\phi}, \underline{x}} := \Rep_q(\check{G})_{\underline{x_1}} \otimes  \Rep_q(\check{G})_{x, \eta_{\phi}},$$
			where $\Rep_q(\check{G})_{x, \eta_{\phi}}$ is the summand in the decomposition of the category $\Rep_q(\check{G}^{\on{sc}})$ over $\pi_1(\check{G})^\vee$ given in \ref{repq(G)}. 
			Then to construct 
			\begin{equation}
				\on{FLE}_{\phi}: \Shv_{\kappa, A_{\Ranp_x}} (\Gra_{G, \rho(\omega_X)})^{LN, \chi}_{\phi} \rightarrow \Gamma^{\lax}(\Ranp_x, \Rep_q(\check{G})_{\eta_{\phi}})
			\end{equation}
			it suffices to show that subcategories $\Shv_{\kappa, A_{x}} (\Gra_{G, \rho(\omega_X)})^{LN, \chi}_{\phi}$ and $\Rep_q(\check{G})_{x, \eta_{\phi}}$ correspond to one another under the equivalence 
			\begin{equation}
				\on{FLE}_x:\Shv_{\kappa, A_{x}} (\Gra_{G_{\ad}, \rho(\omega_X)})^{LN, \chi} \cong \Rep_q(\check{G}^{\on{sc}})
			\end{equation}
			of \cite{GH}. However, by  Remark \ref{reductiontot} (and since the factorization algebras $\Omega$ and algebras $A$ were independent of the center of the group) the statement reduces to the case of tori, where it is evident.
			
			Note that the $\phi$-analog of $\coeff^\loc$ is constructed similarly, and it is also defined and codefined by a kernel. Finally, note that choosing a trivialization of $\phi$ away from x we get $$(\underline{\int_X \Rep_{q}(\check{G})})_{\cG_{\phi}} \cong \Gamma(\Ranp_x, \Rep_q(\check{G})_{\eta_{\phi}}).$$
		\end{proof}
		
		\section{Application: non-vanishing of quantum Whittaker coefficients for semi-simple groups.}
		
		As an application of Theorem \ref{lfm} we are able to relax the condition on the center of $G$ in the results of  \cite{Bwhit}:
		
		\begin{thm}\label{generalgroup}
			For \emph{any} semi-simple reductive group $G$, the functor 
			$$\coeff^{\loc}: \D_{\kappa}(\Bun_{G})^{\temp} \rightarrow \Whit_{\kappa}(G)$$ is conservative.
		\end{thm}
		
		\begin{proof}
			Follows from  \cite[Theorem 6.1.1]{Bwhit}, \cite[Proposition 2.5.4.1(4)]{FR} and Theorem \ref{lfm}.
		\end{proof}

	\typeout{*** ARXIV DEBUG: reading lfunctor2.bbl with real bibitems ***}
	\newcommand{\etalchar}[1]{$^{#1}$}

	\end{document}